\newif\ifpdf
\ifx\pdfoutput\undefined
\pdffalse % we are not running PDFLaTeX
\else
\pdftrue \fi

\newif\iffinal
\finaltrue % Final version

\documentclass[reqno,twoside,11pt]{amsart}
\iffinal\else\usepackage[notref,notcite]{showkeys}\fi

\usepackage{amsmath}
\usepackage{amsfonts}
\usepackage{amssymb}
\usepackage{color}
\usepackage{multido}
\usepackage{pstricks,pst-plot,pstricks-add,pst-math}
\usepackage{enumerate}
\usepackage{verbatim}
\usepackage{epsfig}
\usepackage{setspace}
\usepackage{bbm}

\IfFileExists{epsf.def}{\input epsf.def}{\usepackage{epsf}}

\IfFileExists{myowntimes.sty}{\usepackage{myowntimes}\usepackage{mathrsfs}}
        {\usepackage{mathpazo}\usepackage{mathrsfs}}

\DeclareFontFamily{OT1}{eusb}{} \DeclareFontShape{OT1}{eusb}{m}{n}
{<5> <6> <7> <8> <9> <10> <11> <12> <14.4> eusb10}{}
\DeclareMathAlphabet{\eusb}{OT1}{eusb}{m}{n}

\DeclareFontFamily{OT1}{eusm}{} \DeclareFontShape{OT1}{eusm}{m}{n}
{<5> <6> <7> <8> <9> <10> <11> <12> <14.4> eusm10}{}
\DeclareMathAlphabet{\eusm}{OT1}{eusm}{m}{n}

\DeclareFontFamily{OT1}{eufm}{} \DeclareFontShape{OT1}{eufm}{m}{n}
{<5> <6> <7> <8> <9> <10> <11> <12> <14.4> eufm10}{}
\DeclareMathAlphabet{\mathfrak}{OT1}{eufm}{m}{n}

\DeclareFontFamily{OT1}{fraktura}{}
\DeclareFontShape{OT1}{fraktura}{m}{n} {<5> <6> <7> <8> <9> <10> <11>
  <12> <13> <14.4> [1.1] eufm10}{}
\DeclareMathAlphabet{\fraktura}{OT1}{fraktura}{m}{n}

\DeclareFontFamily{OT1}{cmfi}{} \DeclareFontShape{OT1}{cmfi}{m}{n}
{<5> <6> <7> <8> <9> <10> <11> <12> <13> <14.4> [0.9] cmfi10}{}
\DeclareMathAlphabet{\cmfi}{OT1}{cmfi}{b}{n}

\DeclareFontFamily{OT1}{cmss}{} \DeclareFontShape{OT1}{cmss}{m}{n}
{<5> <6> <7> <8> <9> <10> <11> <12> <13> <14.4> cmss10}{}
\DeclareMathAlphabet{\cmss}{OT1}{cmss}{m}{n}
\newtheoremstyle{thm}{1.5ex}{1.5ex}{\itshape\rmfamily}{}
{\bfseries\rmfamily}{}{2ex}{}

\newtheoremstyle{def}{1.5ex}{1.5ex}{\slshape\rmfamily}{}
{\bfseries\rmfamily}{}{2ex}{}

\newtheoremstyle{rem}{1.3ex}{1.3ex}{\rmfamily}{}
{\itshape}
{} {1.5ex}{}

\theoremstyle{thm}
\newtheorem{maintheorem}{Theorem}

\newtheorem{theorem}{Theorem}[section]
\newtheorem{lemma}[theorem]{Lemma}
\newtheorem{proposition}[theorem]{Proposition}

\newtheorem*{Main Theorem}{Main Theorem.}
\newtheorem{corollary}[theorem]{Corollary}
\newtheorem*{special theorem}{Lindeberg-Feller Theorem for Martingales}

\newtheorem{definition}[theorem]{Definition}
\newtheorem{claim}[theorem]{Claim}

\theoremstyle{rem}

\numberwithin{equation}{section}

\renewcommand{\section}{\secdef\sct\sect}
\newcommand{\sct}[2][default]{%
\refstepcounter{section}
\addcontentsline{toc}{section}{{\tocsection
    {}{\thesection}{\!\!\!\!#1\dotfill}}{}}
\vspace{0.7cm}
\centerline{\scshape\thesection.\ #1} \nopagebreak \vspace{0.2cm}}
\newcommand{\sect}[1]{%
\vspace{0.4cm} \centerline{\large\scshape\rmfamily #1}
\vspace{0.2cm}
}

\renewcommand{\subsection}{\secdef\subsct\sbsect}
\newcommand{\subsct}[2][default]{\refstepcounter{subsection}
\addcontentsline{toc}{subsection}
{{\tocsection{\!\!}{\hspace{1.2em}\thesubsection}{\!\!\!\!#1\dotfill}}{}}
\nopagebreak
{\flushleft\bf
\thesubsection~\bf #1.~}
\noindent
\nopagebreak}
\newcommand{\sbsect}[1]{
\noindent
\textbf{#1.~}
}

\renewcommand{\subsubsection}{%
\secdef \subsubsect\sbsbsect}
\newcommand{\subsubsect}[2][default]{%
\refstepcounter{subsubsection}
\addcontentsline{toc}{subsubsection}{{\tocsection{\!\!}
{\hspace{3.05em}\thesubsubsection}{\!\!\!\!#1\dotfill}}{}}
\nopagebreak
\vspace{0.15\baselineskip} \nopagebreak {\flushleft\rmfamily
\itshape\thesubsubsection
\ \rmfamily #1\/.}\ }
\newcommand{\sbsbsect}[1]{\vspace{0.1cm}\noindent
\rmfamily \itshape
\arabic{section}.\arabic{subsection}.\arabic{subsubsection} \
\sffamily #1\/.\ }

\iffinal

\else

\fi

\renewcommand{\caption}[1]{%
\vglue0.5cm
\refstepcounter{figure}
\begin{minipage}{0.9\textwidth}\small {\sc Figure~\thefigure. }#1\end{minipage}}

\newcommand{\ee}{\end{equation}}
\newcommand{\be}{\begin{equation}}
\newcommand{\eml}{\end{multline}}
\newcommand{\bml}{\begin{multline}}
\newcommand{\ra}{\rightarrow}

\def\qed{ \hfill $\square$}

 \DeclareMathOperator{\E}{\mathbb{E}}
 \DeclareMathOperator{\pr}{\mathbb{P}}

 \newcommand{\mvT}{\boldsymbol{T}}

 \newcommand{\vT}{\mathbf{T}}

\newcommand{\textd}{\text{\rm d}\mkern0.5mu}

\renewcommand{\AA}{\mathcal A}
\newcommand{\BB}{\mathcal B}
\newcommand{\CC}{\mathcal C}
\newcommand{\DD}{\mathcal D}
\newcommand{\EE}{\mathcal E}
\newcommand{\FF}{\mathcal F}
\newcommand{\GG}{\mathcal G}
\newcommand{\HH}{\mathcal H}
\newcommand{\II}{\mathcal I}

\newcommand{\KK}{\mathcal K}

\newcommand{\MM}{\mathcal M}
\newcommand{\NN}{\mathcal N}

\newcommand{\PP}{\mathcal P}

\newcommand{\UU}{\mathcal U}

\newcommand{\WW}{\mathcal W}
\newcommand{\XX}{\mathcal X}

\newcommand{\N}{\mathbb N}

\newcommand{\BbbP}{\mathbb P}

\newcommand{\R}{\mathbb R}

\newcommand{\Z}{\mathbb Z}

\newcommand{\scrD}{\mathscr{D}}
\newcommand{\scrE}{\mathscr{E}}
\newcommand{\scrF}{\mathscr{F}}
\newcommand{\scrG}{\mathscr{G}}

\newcommand{\uo}{\underline{\omega}}

\def\myffrac#1#2 in #3{\raise 2.6pt\hbox{$#3 #1$}\mkern-1.5mu\raise 0.8pt\hbox{$#3/$}\mkern-1.1mu\lower 1.5pt\hbox{$#3 #2$}}

\newcommand{\la}{\langle}
\newcommand{\rra}{\rangle}
\title{Simple Random Walk on Long Range Percolation Clusters II: Scaling Limits}
\author[N. Crawford]{Nicholas Crawford$^1$}
\author[A. Sly]{Allan Sly$^2$}
\thanks{{\tt
    email:allansly@microsoft.com}}
\thanks{{\tt
    email:nickcrawford12345@gmail.com}}
\thanks{NC was supported in
part at the Technion by an Marilyn and Michael Winer Fellowship}
\begin{document}
\thanks{\hglue-4.5mm\fontsize{9.6}{9.6}\selectfont\copyright\,2009 by N.~Crawford and A.~Sly. Reproduction, by any means, of the entire
article for non-commercial purposes is permitted without charge.\vspace{2mm}}
\maketitle
\centerline{$^1$  \textit{Department of Industrial Engineering, The Technion; Haifa, Israel}}
\centerline{$^2$  \textit{Theory Group, Microsoft Research; Redmond, Washington}}
\begin{abstract}
We study limit laws for simple random walks on supercritical long range percolation clusters on $\Z^d, d \geq 1$. For the long range percolation model, the probability that two vertices $x, y$ are connected behaves asymptotically as $\|x-y\|_2^{-s}$.  { When} $s\in(d, d+1)$, we prove that the scaling limit of simple random walk on the infinite component converges to an $\alpha$-stable L\'evy process with $\alpha = s-d$ establishing a conjecture of Berger and Biskup \cite{Berger-Biskup}.  The convergence holds in both the quenched and annealed senses. In the case where $d=1$ and $s>2$ we show that the simple random walk converges to a Brownian motion.  { The proof combines heat kernel bounds from our companion paper~\cite{CS}, ergodic theory estimates and an involved coupling constructed through the exploration of a large number of walks on the cluster.}
\end{abstract}

\section{Introduction}
The study of stochastic processes in random media has been a focal point of mathematical
physics and probability for the past thirty years.
{One such research problem
regards the study of random walk in random environment (RWRE) in its many
forms}. This subject includes tagged particles in interacting particle systems
\cite{K-L}, the study of $\nabla \phi$--fields through the
Helffer--Sj\"{o}strand representation \cite{He, Sj} and random conductance
models.

In this paper we continue the study of {simple random walks (SRW)} on
percolation clusters on the ambient space $\Z^d$.  By now, many properties of
the nearest neighbor percolation model are understood {in the supercritical case}.  We mention in this context the important work of
Kipnis-Varadhan \cite{Kip-Var}, who introduced "the environment viewed from the
particle" point of view to derive annealed functional central limit theorems.
This work was strengthened in Demasi et al. \cite{Demasi} where it was applied
to SRW on nearest neighbor percolation clusters; Sidoravicious-Sznitman
\cite{Sid-Szn} extended the percolation theorem of \cite{Demasi} to the quenched
regime on $\Z^d, d \geq 4$; Remy-Mathieu \cite{RM}, Barlow \cite{Barlow-HKP}
proved quenched heat kernel bounds on supercritical percolation clusters, earlier estimates obtained by Heicklen and Hoffman \cite{HH}; and
finally Mathieu-Pianitskii \cite{Mathieu-Piat} {and}
Berger-Biskup\cite{Berger-Biskup} extended \cite{Sid-Szn} to all { $d\geq
2$}.

We consider a variant of these latter results -- scaling limits for SRW on super
critical Long Range Percolation clusters on $\Z^d$ (LRP).  LRP was first
considered by Schulman in \cite{Schulman83} and Zhang et. al. \cite{Z}.  It is a
random graph process on $\Z^d$ where, independently for each pair of vertices
$x, y \in \Z^d$, we attach an edge $\la x, y\rra$ with probability $p_{x, y}$.
We shall assume an isotropic translation invariant model for the connection
probabilities setting  $p_{x, y} = \texttt{P}(\|x-y\|_2)$ where
\be
\label{Eq;Asymp}
\texttt{P}(r) \sim C r^{-s}
\ee
for some $C \in \R^+$ and large $r$ where $x_n \sim y_n$ denotes $\lim
\frac{x_n}{y_n}\to1$.

The early work concentrated on characterizing when the LRP $(\Omega, \BB, \mu)$
admits an infinite connected cluster in dimension $d=1$.
There are a number of transitions for this behavior of the process as a
function of $s$.  The { first} results of this kind were obtained by Schulman
\cite{Schulman83}: for $s>2$ there is not infinite component unless
$\tt{P}(r)=1$ for some $ r \in \NN$.  Later Newman { and} Schulman
\cite{New-Schul} proved that if $s< 2$,  then in the case that there is no
infinite component, on can adjust (non-trivially) $\tt{P}(1)$ to produce an
infinite component.  They also demonstrated this in the case $s=2$ and $C$
sufficiently large.  Finally, in a striking paper \cite{AN}, Aizenman and Newman
address the case $s=2$, showing that the answer depends on the value of the
constant $C$ in \eqref{Eq;Asymp} ($C=1$ is critical).
%This part is maybe less relevant Moreover the percolation density $M:= \mu(0
%\in \CC^{\infty})$ satisfies $CM^2 \geq 1$ if and only if $M>0$ and that if
%$C>1$, the susceptibility diverges as $\tt{P}(1)$ approaches its critical value
%from below.

More recently, the long range model gained interest in the context of "small
world phenomena", see works such as  \cite{Milgram}, \cite{WattsStrog} and
\cite{BiskupLR} for discussions.  Benjamini and Berger~\cite{BBDiam} initiated a
quantitative study of these models, focusing on the asymptotics of the diameter
on the discrete cycle $\Z/N\Z$.  Their motivation regarded connections to
modeling the topology of the internet, see also \cite{Kleinberg} for a different
perspective. { Further analysis was done in \cite{CGS}.}

{ The study of random walks on LRP clusters was begun in \cite{BLRP}, which
studies recurrence and transience properties of SRW on the infinite component of
supercritical LRP in the general setting where nearest neighbor connections do
not exist with probability $1$.  This paper crucially makes use of the
transience results established therein.  Benjamini, Berger and Yadin \cite{BBYr}
study the spectral gap $\tau$ of SRW on $\Z/ N\Z$, providing bounds of the form
\[
c N^{s-1} \leq \tau \leq C N^{s-1} \log^{\delta}N,
\]
in that case that nearest neighbor connections exist with probability $1$.}

{ In a companion paper to the present paper \cite{CS}, we derive quenched upper
bounds for the heat kernel of continuous time SRW on the infinite component of
supercritical LRP clusters on $\Z^d$ when  $s \in (d, d+2 \wedge 2d)$.  These
estimates are crucial in establishing the quenched limit law of SRW, the main
result of this paper.  The companion paper also yields a number of results on
the geometry of LRP in finite boxes which we make use of  here (see
Section \ref{S:Tech} for details).}

The scaling  exponent of the connection probabilities determines the limiting
behavior of the walk.  Smaller values of $s$ produce more long edges and these
edges determine the macroscopic behavior of the walk suggesting a non-Gaussian
stable law as the limiting process.
To this end we let $\Gamma_\alpha(t)$ denote $d$-dimensional isotropic
$\alpha$-stable L\'{e}vy motion { (formally defined in Section \ref{S:not}).}
We will assume that the percolation process admits an infinite component
$\mu$-a.s and let $\Omega_0$ denote the set of environments where the origin is
in the infinite component with $\mu_0$  the conditional measure on $\Omega_0$.
We now state our main result, a quenched limit law for simple random walk on
long range percolation clusters which affirms a conjecture of Berger and Biskup
\cite{Berger-Biskup} in the case $s\in(d,d+1)$.

%For any $q \in [1, \infty)$, let $L^q([0,1])$ denote the usual Lebesgue space.
\begin{maintheorem}\label{T:Quenched}
Let $d \geq 1$ and $s \in (d, d+1)$.  Let $X_n$ be the simple random walk on
$\omega\in \Omega_0$
and let
\[
X_n(t)=n^{-\frac1{s-d}}X_{\lfloor nt\rfloor}
\]
Then for $\mu_0$-a.s. every environment $\omega$ { and $1\leq q <\infty$}
the law of $(X_n(t),0\leq t \leq 1)$ on $L^q([0,1])$ converges weakly to the law
of an isotropic $\alpha$-stable L\'{e}vy motion with $\alpha=s-d$.
\end{maintheorem}

While the natural topology of convergence to a non-Gaussian stable law is the
Skorohod topology we note that convergence in that sense does not hold. There
exist times at which the walk crosses a particular long edge of the graph an
even number of times on a small time scale.  These events do not appear in the
limit law but do preclude convergence in the Skorohod sense.  We thus {
adopt} the $L^q$ topology as it does not see these spurious discontinuities
{ (see Section \ref{S:Outline} for more details)}.

{ Our proof proceeds by revealing the randomness of the environment as the walk explores the cluster, occasionally encountering macroscopic edges which constitute the main contribution in the limit.  To make this approach rigorous we require the precise heat kernel upper bounds and structural picture established in our companion paper~\cite{CS} together with ergodic theory estimates which guarantee that new vertices are encountered at a constant rate over time.  By combining these estimates with a highly involved coupling and performing this construction simultaneously for a large number of independent walks we establish the quenched convergence.}

\begin{comment}
%We strongly suspect that the Theorem holds for any $d \geq 2, s \in [d+1,
d+2)$, see \cite{CS} for some evidence of this.  Our method of proof takes
advantage of the fact that $d/(s-d)>1$ if $s \in (d,d+1)$ and so does not carry
over.  The conjecture is already stated explicitly in \cite{Berger-Biskup},
though the conjectured range of validity in $s$ contains an error.  Note that
this is not true when $d=1$, see Theorem \ref{T:BM}.  Finally, it is natural to
conjecture that for $d \geq 2, s \geq d+2$, the walk converges to Brownian
Motion.

%Our proof is constructive. Thus the constant $K$ appearing in Theorem
\ref{T:Quenched} is determined microscopically and is related to (among other
things) the rate at which new vertices are visited and the rate at which long
edges are crossed an \textit{odd} number of times.  See Section \ref{S:Limits}
for a more precise description.
\end{comment}

In the case $d=1$ we { establish} a sharp transition in the scaling limit at $s=2$.  In the case $s>2$  there is no infinite component unless $\textrm
P(r)=1$ for some $r \in \N$ \cite{Schulman83}, hence we make the assumption that nearest neighbor edges are included with probability 1.
{ We prove the quenched law converges weakly} to the law of Brownian motion
in the space $C([0,1])$ under the uniform norm.
\begin{maintheorem}\label{T:BM}
Let $d = 1$ and $s >2 $ be fixed.  Assume that \eqref{Eq;Asymp} and { $\textrm
P(1)=1$} hold and let $\omega\in\Omega$.
Let $X_n$ be the simple random walk on $\omega\in \Omega_0$
and let
\[
X_n(t)= \frac1{\sqrt{n}}\left(X_{\lfloor nt\rfloor}+(tn-\lfloor tn
\rfloor)(X_{\lfloor nt\rfloor+1}-X_{\lfloor nt\rfloor})\right)
\]
Then for $\mu$-a.s. environments $\omega$ the law of $(X_n(t),0\leq t \leq 1)$
in $C([0,1])$ converges weakly to  $(K B(t),0\leq t \leq 1)$
where $B(t)$ is standard Brownian motion and $K$ is a constant depending on
the connection probabilities.
\end{maintheorem}

%\begin{remark}
%Combining the previous two theorems gives a reasonably complete picture of the
%possible scaling limits in one dimension, except for the case $d=1, s=2$.  This
%case turns out to be more subtle; even the problem of existence of an infinite
%component depends on tuning the parameter $\beta$ in
%\be
%p_{x,y} = 1- e^{-\beta \|x-y\|^{-s}} \text{ for $\|x-y\|_2 \geq L$}
%\ee
% \cite{AN}.  Besides that work, very little is known and we regard this case as
%the most interesting open topic in LRP on $\Z^d$ (in particular we have nothing
%to say about it here).
%\end{remark}

The rest of the paper is divided as follows.  In Section \ref{S:not}  we
introduce the basic notation used for the remainder of the paper and describe
the ``environment exploration'' process which is at the core of the proof.  The
{ result} relies on a technical coupling construction defined in Section
\ref{S:Coupling}.  { This is the heart of our proof and is described in the proof overview in Section
\ref{S:Outline}.  Section \ref{s:apriori} lists certain a priori bounds needed to show the coupling works with high probability}.    Sections \ref{S:Limits} and \ref{S:Main} detail how the
coupling can be used to derive Theorem \ref{T:Quenched}.   Sections
\ref{S:Ergodic} and \ref{S:Tech} are devoted to justifying various technical
lemmas of Sections \ref{s:apriori}, \ref{S:Coupling}, \ref{S:Limits}, and \ref{S:Main}.  In
particular, to guarantee that our coupling works with high probability, we need
to rule out various rare events and establish ergodic theorems for the walks.
Finally we prove Theorem \ref{T:BM} in Section \ref{S:Sketch}.

\section{Notations, Basic Objects}
\label{S:not}
In this section, we describe the notation used in our proofs.
We denote by $\Omega=\{0,1\}^{\EE}$ the sample space of environments for LRP on
$\Z^d$ where $\EE$ is the edge set of {unordered pairs in $\Z^d$}.
Let $\mu$ denote the product measure on $\Omega$ determined by  connection
probabilities $p_{ij}$ satisfying equation \eqref{Eq;Asymp}.  We assume the
$p_{ij}$ are \textit{percolating}, that is $\mu$ admits an infinite component.
It was proved in \cite{AKN} that in this case the infinite component is unique
and we denote this component by $\CC^{\infty}(\omega)$.  We let $\Omega_0$
denote the subset of environments where the origin is in the infinite component
and let $\mu_0$ be the induced measure of $\Omega_0$,
\[
\mu_0(\cdot)=\mu(\cdot|0\in \CC^\infty(\omega)).
\]
Throughout we assume that $s \in (d, \infty)$ with $d \geq 1$ which ensures
finite degrees almost surely:  letting $d^\omega(x)$ denote the degree of $x$ in
$\omega$, when $s \in (d, \infty)$, $\E_{\mu}[d^{\omega}(0)] < \infty$ and
it follows that $\mu$-a.s., for all $x\in \Z^d$,  $d^{\omega}(x) < \infty$.
{ Let us use the notation
$B_L(v)= \{x \in \Z^d: \|x- v\|_{\infty} \leq L\}$ and henceforth $\alpha=s-d$.}
Here,  we are using the standard notation $\|\cdot\|_p$ for the $\ell^p$ norm on $\R^d$.

The simple random walk on $\omega$ is the walk which moves to uniformly chosen
neighbor of the current location at each step (have transition kernel
$P^{\omega}(x, y)= \frac{\delta(1-\omega_{\la x,y\rra})}{d^{\omega}(x)}$ ).  We
let $(X_i)_{i \in \N}$ denote the random walk trajectory generated by
$P^\omega(x, y)$ with $X_0=0$. For $\ell\in\N$ we let $(X_i^\ell)_{i,\ell\in\N}$
denote independent copies of the walk on the same environment $\omega$.
Studying the joint { \textit{annealed}} law over many $\ell$ plays a crucial role in our proof of the
quenched law.

It will at times be convenient to work under the ``degree-biased'' measure $\nu$
on environments $\Omega$, given by
\[
\nu(A)= \E_{\mu}[\mathbb{1}\{A\} d^{\omega}(0)]/ E_{\mu}[d^{\omega}(0)]
\]
and $\nu_0$, given by
\[
\nu_0(A)=\nu(A|0\in \CC^{\infty}(\omega)).
\]
These measures are important since the process on environments described below
is stationary relative to them. Generically we use the notation $\BbbP$ to
denote the underlying probability distribution and $\E$ the corresponding
expectation.  The actual meaning of this notation should be clear from context.
There is one exception to this rule:  The joint law of
$(\omega, (X^{\ell})_{\ell})$ depends on the distribution on environments, of
which we have the four choices $\mu, \mu_0, \nu, \nu_0$.  To emphasize which
choice is employed, we use subscripts.  Thus we have $\BbbP_{\mu_0}$, $\E_{\nu}$
etc.  By definition, omission of the subscript indicates that we are using the
measure $\mu$.  In Lemma \ref{l:mu} we establish bounds relating these
measures.

Let us now describe the limiting processes.  For $\alpha\in(0,2)$ we recall that
an isotropic $\alpha$-stable L\'{e}vy motion $\Gamma(t): t \in \R^+$ is (up to a
single parameter) the unique c\'{a}dl\'{a}g stochastic process with state space
$\R^d$ having stationary independent increments and the self-similarity property
$\Gamma_{at} \overset{d}{=} a^{1/\alpha} \Gamma_{t}$.

These are non-Gaussian processes whose marginal distributions have power law
tails with index $\alpha$.  If $Y$ is a $\Z^d$ valued random vector such that
$\pr(Y=y)\sim \|y\|_2^{-s}$ for $s=\alpha+d$ then $Y$ is in the domain of attraction
of an isotropic $\alpha$-stable law.  For convenience we will normalize
$\Gamma_\alpha(t)$ so that it is the limit law of associated with random vectors
$Y$ with $\pr(Y=y)\sim \texttt{P}(\|y\|_2)$.  We refer the reader to
\cite{samorodnitskyTaqqu} for more information.

\subsection{Environment Exploration Process}
In Section \ref{S:Coupling}, for each $k \in \N$ we provide a coupling
construction between $\left(\omega, (X_i^{\ell})_{\ell \in k^3, i \in
[2^k]}\right)$ and an i.i.d. family of variables which represent increments of a
family of discrete processes converging to i.i.d. copies $\alpha$-stable Levy
motions. For this purpose it is important to have an alternate description of
the law of  $\left(\omega, (X_i^{\ell})_{\ell \in k^3, i \in [2^k]}\right)$
under $\BbbP_{\mu}$.

The description we use is a version of an environment exploration process in
which the family of walks reveals the edges of the long range percolation
cluster as it encounters new vertices. We emphasize that the reader should be
aware that the standard definition of the environment exploration process is not
the one we use and rather we reveal extra local edges of the process for the
purpose of our coupling.

For each $k$, we let
\[
{ \rho=\rho(k)=k^{-200/(1-\alpha)} 2^{k/\alpha}.}
\]
The quantity
$\rho$ represents the minimum for the macroscopic length scale; indeed, the
contribution to the total variation of $X_i$ from jumps of size less than $\rho$
is negligible under the rescaling by $2^{k/\alpha}$ as $k\rightarrow \infty$
(see Lemma \ref{l:sumSmallJumps}).  Fix $\delta \in (0,1)$ (further restrictions
will be placed on $\delta$ below).
For a vertex $x\in \mathbbm{Z}^d$ let $V_x$ denote the set $ \{ y\in
\mathbbm{Z}^d:\|x-y\|_{\infty} \leq 2^{\delta k} \} $.
For $0\leq i \leq 2^k$ and $1\leq \ell \leq k^3$ define the $\sigma$-algebras
$\FF_{i,\ell}$ inductively as follows:
\begin{itemize}
\item Let $\FF_{0,1}$ be the  $\sigma$-algebra generated by $\{\omega_{0,x}:x\in
\mathbbm{Z}\}$ and $\{\omega_{x,y}:x,y\in V_0 \}$

\item For $1\leq i \leq 2^k$
\begin{align*}
\FF_{i,\ell} &= \FF_{i-1,\ell} \vee \sigma\bigg\{ \{X_i^t\}\cup
\{\omega_{X_i^\ell,y}:y\in \mathbbm{Z}^d\} \cup \{\omega_{x,y}:x,y \in
V_{X_i^\ell} \} \\
&\cup  \{\omega_{x,y} : z \in \mathbbm{Z}^d, x,y \in V_z,
\omega_{X_i^\ell,z}=1,\|X_i^\ell - z\|_{\infty}>\rho \} \\
&\cup \{\omega_{y,z}:y,z\in \mathbbm{Z}^d, \omega_{X_i^\ell,z}=1,\|X_i^\ell -
z\|_{\infty}>\rho  \} \bigg\}
\end{align*}

\item $\FF_{0,\ell}=\FF_{2^k,\ell-1}$ for $\ell
\geq 2$.
\end{itemize}
This $\sigma$-algebra encodes the edges revealed by the process by the first
$\ell-1$ walks and after the $\ell$'th walk reaches the $i$-th step.  It
includes short edges in the surrounding neighborhood of the walk which are used
by the process to determine the coupling.  Denote
\[
\FF_{i,\ell}^- = \FF_{i-1,\ell} \vee \sigma\left\{ \{\omega_{X_i^\ell,y}:|y|\leq
\rho \}\cup \{\omega_{X_i^\ell,y}:y\in \mathbbm{Z}^d\} \cup \{\omega_{x,y}:x,y
\in V_{X_i^\ell} \} \right\} ,
\]
that is ignoring new edges of length greater than  $\rho$.  Let $\WW_{i,\ell}$
denote the set of vertices visited by the first $\ell-1$ walks up to time $2^k$
plus the vertices visited by walk $\ell$ up to time $i$.  Let
\[
\WW^+_{i,\ell}=\WW_{i,\ell}\cup \{x\in \mathbbm{Z}^d: y\in \WW_{i,\ell},
\omega_{x,y}=1,\|x-y\|_{\infty}>\rho\}.
\]

\section{Outline of the Proof}
\label{S:Outline}
Before giving technical details we would like to discuss the important ideas and
difficulties of our approach. The main theorem can be separated into two issues:
identification of the limit law for $X_i$ under $\BbbP_{\mu_0}$ and proof that
this limit law coincides with the limit law for $X_i$ under the quenched measure
$P^{\omega}$ for almost every $\omega \in \Omega_0$.

We use a coupling constructed under the
measure $\BbbP_\mu$ so that edges in $\omega$ are independent. As a consequence
we need to relate results in $\BbbP_\mu$ to those $\BbbP_{\mu_0}$. Note that
this particular difficulty disappears if we make the a priori assumption that
$\texttt{P}(1)=1$. This simplifying assumption will be used in our discussion
although the general case is given in the actual proof.

The power law scaling of the connections probabilities gives the probability of
a long edge according to
\[
\BbbP_\mu(\exists y \in \Z^d :\:\: \|y\|_{\infty}>R, \omega_{\la0,y\rra=1})\sim
R^{\alpha}.
\]
By passing to the degree-biased measure $\BbbP_\nu$ it follows that,
\[
\E_\nu\left[\sum_{i=1}^T\mathbb 1\{\exists y \in \Z^d :\:\: \|y\|_{\infty}>R, \omega_{\la
X_i,y\rra=1}\} \right]\lesssim TR^{\alpha}.
\]
The asymptotic holds for $\BbbP_\mu$ as well. Further, the \textit{a priori}
knowledge that the process is transient, ergodic theory and reversibility imply
that this gives the correct order of magnitude.  In other words,
the largest edges encountered by $X_i$ in time $T$ are $O(T^{1/\alpha})$. This
calculation allows us to determine the right length scales: we expect to see a
non trivial limiting process under the scaling
\[
X_n(t)=n^{-1/\alpha}X_{\lfloor nt\rfloor}.
\]
Our proof will examine time scales of length $n=2^k$ for $k\in \Z^d$ which is
sufficient due to the self-similarity of the limit law.

At this point we make a key use of the assumption that $s\in(d,d+1)$. For any
$n$, let $\rho=\rho(n)=n^{1/\alpha}\log^{-200/(1-\alpha)}n $. We separate the
increments of the walk, $(X_i-X_{i-1})_{i\leq n}$, according to $\rho$. In Lemma
\ref {l:sumSmallJumps}, we show that
\[
n^{-1/\alpha}\sum_{i=1}^n\|{X_i-X_{i-1}}\|_\infty\mathbb
1\{\|{X_i-X_{i-1}}\|_\infty\leq\rho\}\to 0
\]
in probability under $\BbbP_{\mu}$ (giving rates of convergence in the proof).
Thus, if we characterize the behavior of $X_i$ as it encounters edges greater
than $\rho$, then the annealed limit law will follow.

We would like to view the excursions of
$X_i$ near edges of length at least $\rho$ as being asymptotically independent.
Estimates which quantify this claim are stated in Section \ref{S:Coupling} and
are proved in Section \ref{S:Tech}. Moreover, because of transience, on the
macroscopic scale the only excursions which play a role in the limit come from
long edges that are crossed in odd number of times.

Suppose that at step $0\leq i < n$ the walk reaches a vertex $v$ not previously
encountered.  One of the key observations of this paper is that almost all
vertices of distance $\rho$ or greater have not previously been visited by the
path so we can effectively treat the long edges coming out of $v$ as being
chosen according to the connection probabilities $p_{v,y}$ independent of the
past.  We may even treat the local neighborhood of radius $n^\delta$ of a
distant endpoint $y$ as being independent of the past as well provided
$\delta>0$ is sufficiently small.

The observations of the previous paragraph motivate the following analysis of the
walk in the local neighborhood of a long edge. Let $v=X_i$ and $y$ denote the
other end point of the long edge connected to $v$ (there is only one long edge
with high probability). Let $V_v,V_y$ denote the restriction of the percolation
cluster to the balls $B_{n^\delta} (v),B_{n^\delta} (y)$ respectively. Then by
transience, the number of times the walk crosses the edge $(v,y)$ is
approximately (as $n \ra \infty$) the same as the number of times it crosses the
edge before leaving $V_v\cup V_y$.  This latter quantity can be determined by the
degrees of $v$ and $y$ and certain local return times which
measure the chance the walk inside $V_v$ (resp. $V_y$) returns to $v$ (resp.
$y$) before time $n^\gamma$ for a suitably chosen $\gamma>0$.  In order to make use of this information our proof analyses the
number of times the walks encounter a new vertex of given degree and local
return probability.

In Section \ref{S:Coupling} these considerations lead us to construct a coupling
between $X_i$ and a second process $\hat X_i$ having only jumps of length at
least $\rho$ which tracks the displacement of the excursions of $X_i$ which
cross long edges an odd number of times. This construction is at the heart of
the proof and is explained in detail in that section (in fact, the coupling
occurs for $k^3$ i.i.d. copies of $X_i$). The important point in the coupling is
that the increments of $\hat X_i$ are close to discrete versions of the
$\alpha$--stable L\'evy motion. However, the number of increments which
contribute to the position $\hat X_j$ at time $j$ depend on the underlying walk
$X_i$ in a highly non trivial way: through the number of new vertices the walks
encounters of a given degree and local return probability.

To take care of this time dependence we introduce a \textit{third} process
$\mathfrak X_i$. The process $\mathfrak X_i$ uses the same increments as $\hat
X_i$ but is time deterministic. We show that $\mathfrak X_i$ and $\hat X_i$ are
close by showing an ergodic theorem for the number of new vertices of each
degree and return probability type.   Finally,  the increments of $\mathfrak
X_i$ are independent and in the domain of attraction of (a multiple of) the stable variable $\Gamma_\alpha (1)$ and
so this third process converges weakly to $\Gamma_{\alpha}$.

To pass from an annealed limit law to a quenched limit law we will use the law
of large numbers. The idea is to first extend the indicated coupling to $k^3$
walks $(X_i^\ell)_{i\in[2^k],\ell\in[k^3]}$ and apply the Chernoff bound to the
\textit{quenched} law of $\tilde X_k(t)=2^{-k/(s-d)}X_{\lfloor2^kt\rfloor}$. We
will not enter into further details here, except to emphasize that this approach
works precisely because the $k^3$ walks intersect relatively few times. In
particular the observation that the distribution of long edges from a new vertex
is essentially independent of the past still  applies to the exploration process
for $k^3$ walks.  Large deviations estimates now imply that the law of the walk
given $\omega$ converges to the $\alpha$-stable law.

Finally, let us comment on the topology in which the limit law holds. As was
previously mentioned, there will be long edges which the walk $X_i$ crosses in
even number of times in on a short $O(1)$ time scale.  As a result convergence
in the Skorohod topology does not hold.  Instead we prove convergence in the
topology of $L^q([0,1])$ in which these spikes do not affect the limit
law.  An alternative way of dealing with this issue would be to define the limit
as
\[
X^*(t) = n^{-\frac1{\alpha}}X_{m(n)\lfloor \frac{n t}{m(n)}\rfloor}
\]
for some sequence of integers $m(n)\to\infty$ as $n\to\infty$.  This has the
effect of only sampling the process every $m$ steps and since $m$ grows with $n$
the spikes will mostly be between step $mi$ and step $m(i+1)$ for some $i$.  It
is not difficult to modify our proof to show that $X^*(t)$ converges in the
Skorohod topology but we omit this for space considerations.

\section{A priori Estimates}\label{s:apriori}
In order to establish our main coupling we need to bound the probability of several types of unlikely events and also to prove certain ergodic theorems for the number of new vertices encountered by the random walks. The proofs are postponed to Section \ref{S:Tech}.

We define $F^*(\rho, k)$ to be the probability that the first two walks both encounter the same long edge,
\[
F^*(\rho, k)=\big\{\exists x, v\in\mathbb{Z}^d, \: \omega_{x,v}=1, \: \|v\|_\infty \geq \rho, X^1_{i} \text{ and } X^2_j \in \{x,v\} \text{ for some } i, j \in[2^k] \big\}.
\]
The following lemma establishes that this is unlikely which allows us to treat new long edges as essentially being independent.
\begin{proposition}[Pairs of Walks Don't Intersect at Long Edges]
\label{N:41}
There exists $\epsilon>0$ and a constant $c=c(\epsilon)$ such that
\[
\pr_{\mu}\bigg(F^*(\rho, k)\bigg) < c 2^{-\epsilon k}.
\]
\end{proposition}

We now define a number of events involving a single path that we wish to exclude.  For $\gamma, \delta>0$, let
\[
A(\rho)= \{\exists v\in\mathbb{Z}^d, \omega_{0,v}=1, \|v\|_{\infty} > \rho\},
\]
be the event there is a long edge at the origin, let
\[
B(\rho, \gamma, k)= \{ \forall v \in \Z^d : \omega_{0,v}=1, \|v\|_{\infty} > \rho,  \:v\not\in \{X_{1},\ldots,X_{2^{\gamma k+1}}\}\},
\]
be the event that there is a long edge at the origin and the walk does not visit the other end until time $2^{{ \gamma} k+1}$.  We let $C$ denote the event that the walk does not leave the ball of radius $2^{\delta k}$ before time $2^{\gamma k}$
\[
C( \delta, \gamma, k)= \{\max_{0\leq t\leq 2^{\gamma k}} \|X_{t}\|_{\infty} > 2^{\delta k}\}.
\]
We define
\begin{align*}
D(\rho, k)= \{ \exists v \in \Z^d, \: \omega_{0,v}=1, |v| > \rho,  \exists J\in [2^k] \text{ s.t. } X_J=v,  \:(0,v)\not\in \{(X_{i},X_{i+1})\}_{i \leq J}\},
\end{align*}
to be the event that there is a long edge at the origin and the walk reaches the other end of the edge without traversing it.  Next define
\begin{multline*}
E(\rho, \delta, k)=\\
\{ \exists v, x : \|v\|_{\infty}\geq \rho, \min(\|x-v\|_\infty, \|x\|_\infty) \geq 2^{\delta k},\:\: \omega_{0, v}= 1 \text{ and either }  \omega_{x, v}= 1 \text{ or }  \omega_{0, x}= 1\}
\end{multline*}
to be the event that there is a long edge at the origin and one of the endpoints is connected to another edge of length at least $2^{\delta k}$.
Let
\[
F(\rho, \gamma, k)=\big\{\exists { i}, 2^{\gamma k+1}\leq i\leq 2^k, \: \exists v\in\mathbb{Z}^d \text{ s.t.} \: \omega_{0,v}=1 \text{ and } \: \|v\|_\infty \geq \rho, X_{i}\in\{0,v\}\big\},
\]
denote the event that there is a long edge at the origin and the walk returns to either end of the edge at any time after $2^{\gamma k+1}$ steps.
Finally let
\[
G(\rho, \gamma, \delta,k) =  A(\rho) \cap B(\rho, \gamma, k) \cap C(\delta, \gamma,k).
\]
denote the event that there is a long edge at the original and the walk leave the leaves a ball of radius $2^{\delta k}$ in time $2^{\gamma k+1}$ without taking the long edge.

We want to show that these events do not occur at any time up to time $2^k$.  For any event $E \subset \Omega^{\Z}$, let $\mathbf T^{-i} \cdot E= \{\underline \omega: T^{i}\cdot \underline \omega \in E \}$.  For any of the events $O\in \{A, \dotsc, G\}$ defined above, let $O_i:= \mathbf T^{-i} O$ and let
\begin{align*}
\scrG(\rho, \gamma, \delta, k)=& \cup _{i=0}^{2^k} G_i(\rho, \gamma, \delta, k)\\
\scrD(\rho, k)=& \cup _{i=0}^{2^k} D_i(\rho, k)\\
\scrE(\rho, \delta, k)=& \cup _{i=0}^{2^k} E_i(\rho, \delta, k)\\
\scrF(\rho, \gamma, k)=& \cup _{i=0}^{2^k} F_i(\rho, \gamma, k).
\end{align*}
Note that all these events are increasing in $\gamma$.  The following proposition shows we can exclude these events with high probability.
\begin{proposition}\label{c:badEvents1}
For any { $\delta\in (0, 1)$}, there exists $\gamma,\epsilon>0$  such that
\[
\BbbP_{\mu}\left(\scrG(\rho, \gamma, \delta, k) \cup \scrD(\rho, k)\cup \scrE(\rho, \delta, k) \cup \scrF(\rho, \gamma, k)\right)< o( 2^{-\epsilon k}).
\]
\end{proposition}
We note that this estimate is of most interest when $\omega \in\Omega_0$ but holds under the measure $\mu$ as well since if the origin is in a finite component it will be a small component with high probability.

%A simple computation implies
%\be
%\pr\left( E(k, \delta) \right) \leq C 2^{-(1+\delta(s-d)+ o(1))k}.
%\ee
%Lemma \ref{l:mu1} implies the same bound holds under $\BbbP_{\nu}$.
%
%If
%\begin{multline}
%F(k, \delta)=\{\exists i \in [2^k], \exists v, x : \|v\|\geq \rho, \min(\|x-v\|_2, \|x\|_2) \geq 2^{\delta k} \text{ such that }\:\: \omega_{X_i, v}= 1\\ \text{ and either }  \omega_{x, v}= 1 \text{ or }  \omega_{X_i, x}= 1\}
%\end{multline}
%then using stationarity of the walk under  $\BbbP_{\nu}$ and a union bound obtains
%\be
%\pr_{\nu}\left(F(k, \delta)\right) \leq C' 2^{-\delta(s-d)k}
%\ee

\subsection{Ergodic estimates}
{
In this subsection we state the ergodic estimates for the number of new vertices encountered by the paths.  The proofs are given in Section~\ref{S:Ergodic}.
Let $N_i$ be the indicator of the event that $X_i$ is the first visit to that vertex, that is
\[
N_i=\mathbb 1\{X_i \notin \{X_0,\dotsc,X_{i-1}\}\}=\mathbb 1\{\omega_i \notin \{\omega_0,\dotsc,\omega_{i-1}\}\}.
\]

For a vertex $v$ we let $p_v=p_v(\omega)$ denote the quenched probability that a random walk started from $v$ will ever return there which we call the return probability.  The notion of the ``type'' of a vertex and the frequency with which new vertices of each type are encountered plays a crucial role in our proof.  The type is determined by its degree and local return probability.
Let us denote by $C_{q_{j-1}, q_j, m}$ the quantity
\[
C_{q_{j-1}, q_j, m}=\pr_{\nu_0}\left(X_i \neq 0 \text{ for $i > 0$}, p_0(\omega)\in (q_{j-1}, q_j),\textrm{d}^{\omega}(0)=m \right)
\]
and let
\[
C^\star=\pr_{\nu_0}\left(X_i \neq 0 \text{ for $i > 0$}\right)
\]
the annealed escape probability.  By reversibility, this is also the rate at which a walk sees new vertices when started in the infinite component. Our main ergodic estimate shows that they indeed give the long run frequency of new vertices of a particular type.
Denote the total number of new vertices by $N_t := \sum_{i=1}^t N_i$ and for any measurable subset $\AA\subseteq [0,1]$, $\MM \subseteq \N$  denote the number of vertices of a given type by
\[
N^{\AA,\MM}_t := \sum_{i=1}^t N_i \mathbbm{1}\left\{ p_{X_i}\in\AA, \textrm{d}^\omega(X_i)\in\MM \right\}~.
\]
For the $\ell$-th walk we denote these by $N^\ell_t$ and $N^{\ell,\AA,\MM}_t$.
Let $L(\AA,\MM) \in \Omega^{\Z}$ denote the event
\[
L(\AA,\MM):=\left\{\underline \omega: \omega_0 \notin \{\omega_i: i \leq -1\}, p(\omega_0) \in\AA,\textrm{d}(\omega_0)\in\MM \right\}.
\]
The following events require that the number of new vertices up to time $t$ for all $1\leq t \leq 2^k$ is close to what we expect.  Let $\HH_{k,\chi}^{\ell}$ denote the event that
\[
 \sup_{\AA,\MM} \max_{1\leq t \leq 2^k} \left| N^{\ell,\AA,\MM}_t - t \pr_{\nu_0}(L(\AA,\MM)) \right| \leq \chi 2^k \,
\]
and let $\HH_{k,\chi}$ denote the event
\[
 \bigg\{ \frac1{k^3}\sum_{\ell=1}^{k^3} \mathbbm{1}\{ \HH_{k,\chi}^{\ell} \} > 1 - \chi \bigg \}.
\]
Roughly this event says that the Ergodic Theorem bound holds by time $2^k$ for most of the independent copies of the random walk.  The following lemma shows that $\HH_{k,\chi}$ holds for all but finitely many values of $k$ almost surely.

\begin{lemma}\label{l:fullErgodic}
For any $\chi>0$,
\[
\lim_{k'\rightarrow\infty} \pr_{\mu_0}\left(\cap_{k \geq k'} \HH_{k,\chi}\right) =1.
\]
\end{lemma}

}

\section{Main Coupling}
\label{S:Coupling}

Let $(X^\ell_i)_{\ell,i\in \mathbb{N}}$ denote independent copies of the random walk started from the origin with respect to the same environment $\omega$.
In this section we define and analyze the main tool of our proof, a coupling of $(X^{\ell}_i)_{\ell=1}^{k^3}$ to a new sequence of walks $(\hat X_i^\ell)_{\ell=1}^{k^3}$ on $\Z^d$ which will be simpler to handle.  The coupling involves a complicated bookkeeping of the long edges each the walks in the process encounters and the number of times these edges are traversed.  In particular when the process arrives at a vertex which has an edge of length greater than $\rho$ we keep track of the edge's size as well as the degree and local escape probabilities of each end of the edge.

Recall that $p_v=p_v(\omega)$ denotes the return probability the vertex $v$ in the environment $\omega$ and let $(\PP,\DD)$ denote the  distribution of the random vector $(p_0(\omega), \textrm{d}^{\omega}(0))$ under $\BbbP_{\mu}$.

For the remainder of our proof, it is important that $1+\delta d< d/(s-d)$, so that each walk history, thickened by a neighborhood of radius $2^{\delta k}$, takes up volume at most $ 2^{(1+\delta d)k}$.  Because of the power law tail distribution, endpoints of new long edges can occur essentially anywhere in a region of volume $2^{k d/(s-d)} \gg 2^{(1+\delta d)k}$ and are very unlikely to be close to the history of the walk.  For definiteness, let us fix, for the remainder of our proof, $$\delta= 1/2\left(\frac {1}{s-d}-1\right)\wedge1/2.$$  We may then choose $\gamma$ sufficiently small so that the results of Section \ref{s:apriori} hold for the pair $(\delta, \gamma)$.

For a vertex $v$ and each $k \in \N$ denote $\tilde p_v=\tilde p_v(k)$ the probability that a walk started from $v$ and conditioned to stay in the set $\{u:\|v-u\|_{\infty}<2^{k\delta}\}$ returns to $v$ before time $2^{k\gamma}$ and set to 1 if $v$ has no neighbours within distance $2^{k\delta}$ { which we call the local return probability.}

{ As mentioned in Section \ref{S:Outline} we set $\rho=\rho_k=2^{\frac{k}{s-d}}/k^{200/(1-\alpha)}$ as our threshold for long jumps.}
We let $\tilde{\textrm{d}}^\omega(v):= \#\{u:\|v-u\|_\infty\leq { 2^{\delta k}}\}$ and let $(\PP(k),\DD(k))$ denote the joint distribution of $(\tilde p_v,\tilde{\textrm{d}}^\omega(v))$ { of the origin under $\BbbP_{\mu}$.
We observe} that $(\PP(k),\DD(k))$ converges in distribution to $(\PP,\DD)$.  We will prove a much stronger statement in Lemma \ref{l:localEscapeProb}.

For technical reasons, we discretize the the joint distribution $(\PP,\DD)$ as follows. For each positive integer $J$ we choose a sequence $0=q_0<q_1<\ldots<q_{J} <1$ so that the distribution $\PP$ does not have any atoms on the $q_i$ and so that for some sequence $\psi_J$ converging to $0$ we have that $\frac1{1-q_i}-\frac1{1-q_{i+1}}< \psi_J$ and
\be
\label{discrete}
\bar{C} = C^\star - \sum_{j=1}^J \sum_{m=1}^J C_{q_{j-1},q_j,j}
\ee
for all $J$.

\subsection{Coupling Variables}

To define the coupling, we introduce several sequences of random variables.
For each $i \geq 1, (j ,m) \in [J]^2\cup \{(0,0)\}$ and $\ell\in[k^3]$ and $x\in \mathbbm{Z}^d$ define $w_i^{\ell,j,m}(x)$ as independent Bernoulli random variables with probability
\[
\pr\left(  w_i^{\ell,j,m}(x) =1 \right)=\texttt{P}(\|x\|_2) %\mathbbm{1}_{\{|x|>\rho\}}.
\]
Also define $w_i(x)$ as independent Bernoulli random variables with the same probabilities.  These random variables will be coupled with the newly revealed edges found by the exploration process.

We will denote the geometric distribution by $\pr(\mathrm{Geom}(p)=r)= \left(1-p\right)^{r}p$ for $r=\N_0.$
Let $U_0,U_1,\ldots$ be an i.i.d. sequence of uniform $[0,1]$ random variables and define a geometric process as
\begin{equation}\label{e:geomProcess}
R(t)=\min\{i\geq 0: U_i <  t\}.
\end{equation}
Then $R(t)$ is a decreasing integer valued stochastic process  on $[0,1)$ with marginals given by $\mathrm{Geom}(t)$.
For each $i,\ell \geq 0$ and $(j,m)\in [J]^2\cup \{(0,0)\}$ let $R_i^{\ell,j,m}(t)$ and $\tilde R_i^{\ell,j,m}(t)$ be independent copies of $R(t)$. These processes will be used to decide how many times the walk returns to a vertex before escaping and never returning.

Finally for each $i\geq 0, (j ,m) \in [J]^2\cup \{(0,0)\}$ and $\ell\in[k^3]$ let $(\mathfrak r_i^{\ell,j,m},\mathfrak d_{i}^{\ell,j,m})$  independent variables, distributed as $(\PP(k),\DD(k))$.  When a new long edge is encountered by a walk of the process, the local neighbourhood of the other side of the long edge will generally be independent of the walk so far and will be coupled with the $(\mathfrak r_i^{\ell,j,m},\mathfrak d_{i}^{\ell,j,m})$.

\subsection{Coupling Construction}

Using the random variables defined above we now show how we couple the sequence of $k^3$ walks with  { the  variables of the previous subsection.  In the following Section we will then use these to couple the walks to a family of processes $(\hat X_i^\ell)$.} We will reveal the edges of the graph either as the the sequence of walks encounters them or if they are in some local neighborhood of the vertices the walks do encounter.  The key point will be to understand the behavior of each walk after it encounters a long edge of size greater than $\rho$.  Thus we define the coupling construction by two sets of rules: one which will be used for most of the walk and a second special phase which begins when a new long edge is encountered and which then runs for $2^{k\gamma+1}$ steps.

As part of our coupling process we will define several auxiliary ``flag'' variables to track certain events through the coupling. Roughly, they are described as follows:
\begin{itemize}
\item $\AA_{i,\ell}$ represents that in step $i$ of walk $\ell$ a new long edge has been encountered.
\item $\AA^*_{i,\ell}$ will indicate the phase that the walk is in, with a value 1 indicating that we are in the special phase and have recently encountered a long edge.
\item $\BB_{i,\ell}$ represents that one of several types of rare events occurred which we loosely describe as an ``error.''
\item $\NN_i^{\ell,j,m}$ for $(j ,m) \in [J]^2\cup \{(0,0)\}$ will denote a whether a new vertex has been found and what its local return probability $\tilde p_v$ and local degree $\tilde{\textrm{d}}^\omega (v)$ are.
\end{itemize}

\noindent More precisely, we set $\AA^*_{i-1,\ell}$ to be the indicator of the event that for some $i - 2^{\gamma k+1} -1\leq i' < i-1$ that $\AA_{i',\ell}=1$.  This will indicate which phase we are to use.  Let $v$ denote $X_{i-1}^\ell$.

\noindent \textbf {Main Phase:}
In this case $\AA^*_{i-1,\ell}=0$.  We now describe how a new step from { $v=X_{i-1}^\ell$} to $X_{i}^\ell$ is chosen.

{\textit{ Case 1 (already visited vertices):} Suppose that $v \in \WW^+_{i-2,\ell}$.

By definition, in this case one of the first $\ell$ walks has already visited $v$ or has visited a neighbour of $v$ at distance more than $\rho$.  Either way the entire neighbourhood of $v$ has already been revealed and the walk chooses $X_{i}^\ell$ uniformly amongst the neighbours of $v$.
\begin{itemize}
\item Set $\AA_{i-1,\ell}=0$
\item Set $\BB_{i-1,\ell}=1$ if there exists $y\in \mathbbm{Z}^d$ with $\omega_{v,y}=1$ and $|v-y|>\rho$ which we will call a \textit{type one} error. Otherwise set $\BB_{i-1,\ell}=0$.
\item Set $\NN_{i-1}^{\ell,j,m} = 0$ for all $(j ,m) \in [J]^2\cup \{(0,0)\}$.
\end{itemize}

\begin{claim}\label{cl1}
There exists $\epsilon>0$ such for each $\ell$ we have that $P(\exists 0\leq i \leq 2^k-1: \BB_{i-1,\ell}=1) = o(2^{-k\epsilon})$.
\end{claim}

\begin{proof}
This event occurs if at some time $i\in [2^k-1]$, a long edge $e$ of length greater than $\rho$ is encountered in the main phase which has already been encountered by walk $\ell$ or one of the previous walks.  If it were previously encountered by this walk (and not by a previous one) then it must have started a special phase (see below) and so have been encountered some time before time $i-2^{\gamma k}$.  The probability of this event is bounded by Proposition \ref{c:badEvents1}.
{ The bound is completed using Proposition \ref{N:41} and a union bound to account for possible intersections of different walks at long edges.}
\end{proof}

{\textit{ Case 2 (new vertices):}  Suppose next that $v \not\in \WW_{i-2,\ell}$ and so we are at a new vertex.
First reveal all edges not already revealed in $$ \{\omega_{v,y}:\|y\|_{\infty}\leq \rho \}\  \hbox{ and }  \{\omega_{x,y}:x,y\in \mathbbm{Z}^d,\|x-v\|_\infty \leq 2^{\delta k},\|y-v\|_\infty \leq 2^{\delta k} \}$$ so that we have revealed the $\sigma$-algebra $\FF_{i-1,\ell}^-$.  Both $\tilde p_v$ and $ \tilde{\textrm{d}}^\omega(v)$ are $\FF_{i-1,\ell}^-$ measurable.
If $\tilde p_v\leq q_J$ and ${\tilde{\textrm{d}}^\omega} (v)\leq J$ then set $j=\min\{j:q_j>p_v\}$ and $m=\tilde{\textrm{d}}^{\omega} (v)$, otherwise set $j=0,m=0$. Set $\NN_{i-1}^{\ell,j,m} = 1$ and the other $\NN_{i-1}^{\ell.\cdot,\cdot}$ to 0.

We now reveal the possible long edges from $v$ by coupling them to the $w$ random variables.  Let $$\iota=\phi_i^{\ell, j, m}=\sum_{i'=0}^{i-1} \NN_{i'}^{\ell,j,m},$$ { count} the number of distinct \textit{new} vertices of type $(j,m)$ have been encountered so far by the $\ell$'th walk. We note the use of two notations $\iota, \phi_i^{\ell, j, m}$;  our use of $\iota$ implicitly depends on the fixed triple $(\ell, j, m)$.

For each $x\not\in \WW^+_{i-2,\ell}$ with $\|x-v\|_\infty>\rho$ we have that $\omega_{v,x}$ is so far unrevealed.  We couple the environment and the random variables $w_\iota^{\ell,j,m}$ so that for each such $x$, $\omega_{v,x}=w_\iota^{\ell,j,m}(x-v)$.  Our procedure now depends on whether or not any of the $w_\iota^{\ell,j,m}(x-v)$ are non-zero:
\begin{enumerate}
\item \textbf{$\sum_{x:\|x-v\|_\infty > \rho}w_\iota^{\ell,j,m}(x-v)=0$}; then we do not encounter a new long edge.  Set $\AA_{i-1,\ell}=0$ and $\BB_{i-1,\ell}=0$ and choose the next step of the walk $X_i^\ell$ uniformly from the neighbours of $v$.
\item \textbf{$\sum_{x:\|x-v\|_\infty > \rho}w_\iota^{\ell,j,m}(x-v) \geq 2$ or $\sum_{x:\|x-v\|_\infty > \rho}w_\iota^{\ell,j,m}(x-v)=1$ and there exist $x,y$ such that $\|x-y\|_\infty \leq 2^{\delta k+1}$, $w_\iota^{\ell,j,m}(x-v)=1$ and $y\in \WW^+_{i-2,\ell}$.}  Both of these are unlikely and we will call them  \textit{type two} errors.  Set $\AA_{i-1,\ell}=0$ and $\BB_{i-1,\ell}=2$ and choose $X_i^\ell$ uniformly from the neighbours of $v$.
\item  \textbf{Everything else:} In the remaining case we have some  $x$ with $\|x-v\|_\infty>\rho$ and $\omega_{v,x}=1$.  Moreover, the set of edge indicator variables
$\{\omega_{y,z}:y,z \in V_{x} \}$ are so far unrevealed and are therefore independent of the construction up to this point.  The choice of edges here determines $(\tilde p_x,\tilde {\textrm{d}}^{\omega} (x))$ which are distributed according to $(\PP(k),\DD(k))$.  We reveal the edges $$\{\omega_{y,z}:y,z \in V_{x} \}$$ and couple them so that $(\tilde p_x, \tilde {\textrm{d}}^\omega(x))=(\mathfrak r_\iota^{\ell,j,m},\mathfrak d_{\iota}^{\ell,j,m})$.  We also reveal the remaining edges in $\{\omega_{x,y}:y\in\mathbbm{Z}^d\}$.  If $ \textrm{d}^\omega(v)\neq \tilde {\textrm{d}}^\omega(v)+1 $ or $\textrm{d}^\omega(x)\neq  \tilde {\textrm{d}}^\omega(x)+1$ set $\BB_{i-1,\ell}=3$, otherwise set  $\AA_{i-1,\ell}=1$ and $\BB_{i-1,\ell}=0$
\end{enumerate}

\begin{claim}\label{cl2}
With
\[
\delta=1/2\left(\frac {1}{s-d}-1\right)\wedge1/2,
\] there exists $\epsilon>0$ such that
\[
P(\exists 0\leq i \leq 2^k-1, \exists \ell \in [k^3]: \BB_{i-1,\ell} =2) = o(2^{-k\epsilon}).
\]
\end{claim}

\begin{proof}
First note that the decay of the probabilities of the $w_i(z)$ gives that $\pr(\sum_{z:\|z\|_\infty > \rho}w_i(z) \geq 2)\leq 2^{-2k(1-o(1))}$.  By a union bound, this implies that we never have $\sum_{x:\|x-v\|_\infty > \rho}w_\iota^{\ell,j,m}(x-v) \geq 2$ except with  probability $o(2^{-k\epsilon})$.
A similar analysis shows that none of the $k^3$ walks ever encounters a vertex with two long edges except with probability $o(2^{-k\epsilon})$.

By definition, for any $i,\ell$ we have that $|\WW_{i,\ell}|\leq k^3 2^k$.  Now if none of the walks ever reaches a vertex with two long edges then for any $i$ and $\ell$ we have that $|\WW^+_{i,\ell}|\leq 2k^3 2^k$.
\[
\pr\Bigg(\sum_{{ \stackrel{x:\|x-v\|_{\infty} > \rho}{\min_{y\in \WW^+_{i,\ell}}  \|x-y\|_\infty \leq 2^{k\delta+1}}}}w_{i}^{\ell,j,m}(x-v) \geq 1\Bigg) \leq \texttt{P}(\rho) 2^{1+d(k\delta+2)}k^3 2^k = o(2^{-k(1+\epsilon)})
\]
for some fixed $\epsilon>0$.  By our choice of $\delta$, this holds essentially because the walks  only explore a vanishing proportion of the local area on the length scale $\rho$.  A union bound completes the proof.
\end{proof}

\begin{claim}\label{cl3}
With
\[
\delta=1/2\left(\frac {1}{s-d}-1\right)\wedge1/2,
\]  there exists $\epsilon>0$ such that
\[
\BbbP(\exists 0\leq i \leq 2^k-1, \exists \ell\in [k^3]: \BB_{i-1,\ell} =3) = o(2^{-k\epsilon}).
\]
In fact, we may take $\epsilon = \delta(s-d) - o(1)$.
\end{claim}

\begin{proof}
For the event $\{\BB_{i-1,\ell} =3\}$ to take place the walk must encounter a vertex $v$ at time $i$ with an edge $(v,u)$ of length at least $\rho$ such that either $v$ or $u$ is attached to another edge of length at least $2^{\delta k}$ which implies that the event $\scrE(\rho, \delta, k)$ takes place and hence the bound follows by Proposition \ref{c:badEvents1}.
\end{proof}

\noindent \textbf{Special phase: Coupling procedure after encountering a new long edge.}

\noindent We now describe the more complicated coupling after a long edge is encountered (when $\AA_{i-1,\ell}=1$).  At such an event, the walk is at the vertex $v=X_{i-1}^\ell$, which is connected to a vertex $x$ such that $\|v-x\|_{\infty}>\rho$.  Our coupling ensures that
\[
(\tilde p_x,\tilde{\textrm{d}}^\omega(x))=(\mathfrak r_\iota^{\ell,j,m},\mathfrak d_{\iota}^{\ell,j,m})
\]
where $\iota=\sum_{i'=1}^{i-1} \NN_{i'}^{\ell,j,m}$.  For the rest of this subsection, if $\tilde p_v\leq q_J$ and $\tilde{\textrm{d}}^\omega (v)\leq J$ then we denote $j=\min\{j:q_j>p_v\}$ and $m=\tilde{\textrm{d}}^\omega (v)$ and otherwise $j=0,m=0$.

If the walk is in the infinite component, then transience implies it will cross the edge $(v,x)$ a finite number times but will then escape and never return to the local neighbourhood.  In the scaling limit the crucial information will be the parity of the number of times the walk crosses the edge.
In our coupling this will be determined by the geometric processes $R_\iota^{\ell,j,m}(t)$ and $\tilde R_\iota^{\ell,j,m}(t)$ (where $\iota$ represents the number of vertices of type $(j, m)$ the $\ell$'th walk has encountered by time $i-1$).

Let $V^*$ denote the graph with vertices $V_v\cup V_x$ and edges
\[
\{(v,x)\}\cup \{(y,z):y,z\in V_v,\omega_{y,z}=1\} \cup \{(y,z):y,z\in V_x,\omega_{y,z}=1\}
\]
and let $Y_t$ denote a random walk on $V^*$ started at $v$.  Recall our choice $\delta=1/2\left(\frac {1}{s-d}-1\right)\wedge1/2$.  For $\gamma$ depending on $\delta$ as in Section \ref{s:apriori}, let
\[
\tau^*=\inf\{t>2^{\gamma k}:\forall t-2^{\gamma k}\leq t' \leq t, Y_t\not\in \{v,x\}\}~,
\]
that is, $\tau^*$ is the first time that the walk $Y_t$ has not been at either $v$ or $x$ in the last  $2^{\gamma k}$ steps.  What we want to know, is at time $\tau^*$ which side of $V^*$ will $Y_t$ end up on, (i.e. is $V_{\tau^*}$ in $V_v$ or $V_x$)?

Since $(v,x)$ is the only edge between $V_v$ and $V_x$ a walk starting from $v$ can do one of three types of excursions:
\begin{enumerate}
\item Move to $x$ with probability $\frac1{1+\tilde{\textrm{d}}^{\omega}(v)}$ (resp. $\frac1{1+\tilde{\textrm{d }}^{\omega}(x)}$).
\item Move to another vertex in $V_v$; then perform a walk in $V_v$ and return to $v$ in with in the next $2^{\gamma k}$ step with probability $\frac{\tilde p_v \tilde{\textrm{d }}^{\omega}(v)}{1+\tilde{\textrm{d }}^{\omega} \d(v)}$.
\item Move to another vertex in $V_v$; then perform a walk in $V_v$ and not return to $v$ in the next $2^{\gamma k}$ steps with probability $\frac{(1-\tilde p_v) \tilde{\textrm{d }}^{\omega}(v)}{1+\tilde{\textrm{d }}^{\omega}(v)}$.
\end{enumerate}
Let $R_v$ be the number of excursions of type (1) made by the walk from $v$ before it makes an excursion of type (3) from $v$.

Analogous statements hold for walks started from $x$.  Let $R_x$ be the number of excursions of type (1) from $x$ made by the walk from the first time it visits $x$ before it makes an excursion of type (3) from $x$. The following claim is immediate from the definitions.

\begin{lemma}
The random variables $R_v$ and $R_x$ are independent and distributed respectively as $$\mathrm{Geom}\left(\frac{(1-\tilde p_v) \tilde {\textrm{d}}^\omega(v)}{1+(1-\tilde p_v)  \tilde{\textrm{d}}^\omega(v)}\right) \hbox{ and } \mathrm{Geom}\left(\frac{(1-\tilde p_x)  \tilde{\textrm{d}}^\omega(x)}{1+(1-\tilde p_x)   \tilde{\textrm{d}}^\omega(x)}\right) .$$   Moreover, if $R_v > R_x$ then $Y_{\tau^*}\in V_x$ while if $R_v\leq R_x$ then  $Y_{\tau^*}\in V_v$.
\end{lemma}

\begin{proof}
The type of excursion depends only on steps from $V_v$ and $V_x$ respectively and thus $R_v$ and $R_x$ are independent.
By  definition  the time $\tau^*$ occurs in the first type 3 excursion and hence $R_v$ and $R_x$ determines which side the walk is on at time $\tau^*$.
Note that the asymmetry between $R_v > R_x$ and $R_v\leq R_x$ comes from the fact that the walk starts at $v$.
\end{proof}

So we may couple the random walk $Y_t$ to the process (and in particular $X^{\ell}$) so that
$$R_v=  R_\iota^{\ell,j,m}\left(\frac{(1-\tilde p_v)  \tilde {\textrm{d}}^\omega(v)}{1+(1-\tilde p_v)
\tilde{\textrm{d}}}^{\omega(v)}\right),\qquad R_x=  \tilde R_\iota^{\ell,j,m}\left(\frac{(1-\tilde p_x)  \tilde{\textrm{d }}^{\omega}(x)}{1+(1-\tilde p_x) \tilde{\textrm{d }}^{\omega}(x)}\right).$$

Now construct the random walk step by step from $X_{i-1+t}^\ell$ to $X_{i+t}^\ell$ for $0 \leq t \leq 2^{\gamma k +1}-1$ as follows:
\begin{enumerate}
\item If $X_{i-1+t}$ is a vertex not already visited then reveal any unrevealed edges in the set $$\{\omega_{X_i^\ell,y}:y\in \mathbbm{Z}^d\} \cup \{\omega_{x,y}:x,y \in V_{X_i^\ell} \}.$$
\item If there exists $y\in \mathbbm{Z}^d$ with $\omega_{X_{i-1+t}^\ell,y}=1$ and $\|X_{i-1+t}^\ell-y\|_\infty>\rho$ and $\{X_{i-1+t}^\ell, y\}\neq \{v,x\}$ then set $\BB_{i-1+t,\ell}=4$.
\item Choose $X_{i+t}^\ell$ uniformly amongst the neighbours of $X_{i-1+t}^\ell$. If $X_{i-1+t}^\ell=Y_t$ and the edge $(X_{i-1+t}^\ell,X_{i+t}^\ell)$ is in the graph $V^*$ then couple so that $X_{i+t}^\ell=Y_{t+1}$.
\end{enumerate}
Let
\[
\tau=\begin{cases} 2^{\gamma k +1} & \forall 1\leq t \leq 2^{\gamma k +1}-1,  X_{i+t}^\ell =  Y_t\\
\min\{1\leq t < 2^{\gamma k +1}: X_{i+t}^\ell\neq Y_t\} &\hbox{otherwise,}\end{cases}
\]
and set $\BB_{i-1+\tau,\ell}=5$ if $\tau \neq 2^{\gamma k +1}$.
Note that conditional on $X_{i-1+t}^\ell$ staying within $V^*$ up to time $i+2^{\gamma k +1}$ and it making no steps between $V_v$ and $V_x$ except along the edge $(v,x)$ then $\tau=2^{\gamma k +1}$.

Let $\KK$ denote the event that $\tau=2^{\gamma k +1}$, that $\tau^*< 2^{\gamma k +1}$ and that  $Y_t\not\in\{v,x\}$ for all $\tau^*\leq t \leq 2^{\gamma k +1}$.  If $\KK$ does not hold set $\BB_{i-1+2^{\gamma k +1}}^\ell=6$.  { Finally, inside the special phase we set all $\NN^{\ell,j,m}_{i-1+t}=0$, $\AA_{i-1+t,\ell}=0$ and set $\BB_{i-1+t,\ell}=0$ unless otherwise stated.  By definition $\AA_{i-1+t,\ell}^*=1$ inside the special phase.
}

The following lemma follows immediately by definition from the construction.

\begin{lemma}
On the event $\KK$ we have that $X_{i-1+2^{\gamma k +1}}^\ell\in V_v$ if $R_v > R_x$ and  $X_{i-1+2^{\gamma k +1}}^\ell\in V_x$ if $R_v \leq R_x$.
\end{lemma}

We now bound the probability of errors of type 4,5 or 6 (i.e. $\BB_{i}^{\ell} \in \{4, 5, 6\}$).

\begin{claim}\label{cl456}
With
\[
\delta=1/2\left(\frac {1}{s-d}-1\right)\wedge1/2,
\] there exists $\gamma, \epsilon>0$ so that we have
\[
P(\exists 0\leq i \leq 2^k-1 \exists \ell \in [k^3]: \BB_{i-1,\ell} = \{4,5,6 \}) = o(2^{-k\epsilon}).
\]
\end{claim}

\begin{proof}
For an error of type 4 to occur the walk must encounter 2 long jumps of length at least $\rho$ within $2^{\gamma k +1}$ steps.  For an error of type 5 to occur it must encounter a long edge and then leave the $2^{\delta k}$ neighbourhood of that edge in time less than $2^{\gamma k +1}$.  Finally an error of type 6 to occur implies that the walk returns to a long edge after an excursion of at least $2^{\gamma k}$.  The probability of each of these events is bounded by Proposition \ref{c:badEvents1} establishing the claim.
\end{proof}

This completes the coupling.  We denote by $\mathcal{G}$ the event that no errors occurred in the coupling (i.e. $\sum_{\ell=1}^{k^3} \sum_{i=0}^{2^k} \BB_{i}^\ell=0$).  Combining Claims \ref{cl1}, \ref{cl2}, \ref{cl3} and \ref{cl456} we have the following result.
\begin{lemma}\label{l:couplingBound}
There exists $\epsilon>0$ such that $P(\GG) \geq 1 - O(2^{-\epsilon k})$.
\end{lemma}

We now use this coupling as the basis for establishing the scaling limit.

\section{Limiting Processes}
\label{S:Limits}
{ Throught this section we use the notation $|x|= \|x\|_\infty$ for $x \in \R^d$.}
%Let $Z_{t}$ denote symmetric $\alpha$-stable L\'{e}vy motion defined in the introduction.
\subsection{Approximating $(X_i^\ell)_{i, \ell}$ by $(\hat{X}_i^\ell)_{i, \ell}$}
%for $0\leq i \leq \phi_{2^k}^{\ell, j,m}$
{ For the first approximation step we construct a process $(\hat{X}_i^\ell)_{i, \ell}$ and compare this family to the underlying walks directly.}  We emphasize that the bounds obtained in this coupling \textit{do not} require that $0$ is in the infinite component.  { Indeed, when 0 is not in the infinite component then its component will with high probability have diameter less than $\rho$ and so both processes $(X_i^\ell)_{i, \ell}$ and $(\hat{X}_i^\ell)_{i, \ell}$ will be close to $0$ under spatial rescaling by $2^{-k/(s-d)}$.}

For $i\in [2^k], (j ,m) \in [J]^2\cup\{(0,0)\}$ and $\ell\in[k^3]$ define
\[
\sigma_i^{\ell, j,m}= \mathbbm{1}\left\{ R_i^{\ell,j,m}\left(\frac{(1- \tilde{p}_{X_{i-1}^{\ell}}) m}{1+(1- \tilde{p}_{X_{i-1}^{\ell}})m}\right) > \tilde R_i^{\ell,j,m}\left(\frac{(1-\mathfrak r_i^{\ell,j,m}) \mathfrak d_{i}^{\ell,j,m}}{1+(1-\mathfrak r_i^{\ell,j,m}) \mathfrak d_{i}^{\ell,j,m}}\right) \right\}
\]
which determines the side of $V^*$ the walk ended on and in particular whether it contributes to the scaling limit.
Also define
\[
\sigma_i^{+,\ell, j,m}= \mathbbm{1}\left\{ R_i^{\ell,j,m}\left(\frac{(1- q_j) m}{1+(1- q_j)m}\right) > \tilde R_i^{\ell,j,m}\left(\frac{(1-\mathfrak r_i^{\ell,j,m}) \mathfrak d_{i}^{\ell,j,m}}{1+(1-\mathfrak r_i^{\ell,j,m}) \mathfrak d_{i}^{\ell,j,m}}\right) \right\}
\]
and
\[
\sigma_i^{-,\ell, j,m}= \mathbbm{1}\left\{ R_i^{\ell,j,m}\left(\frac{(1- q_{j-1}) m}{1+(1- q_{j-1})m}\right) > \tilde R_i^{\ell,j,m}\left(\frac{(1-\mathfrak r_i^{\ell,j,m}) \mathfrak d_{i}^{\ell,j,m}}{1+(1-\mathfrak r_i^{\ell,j,m}) \mathfrak d_{i}^{\ell,j,m}}\right) \right\}
\]
for $i\in [2^k], (j ,m) \in [J]^2$ and $\ell\in[k^3]$.
{ Recall that if $\tilde p_v\leq q_J$ and $\tilde{\textrm{d}}^\omega (v)\leq J$ then we denote $j=\min\{j:q_j>p_v\}$ and $m=\tilde{\textrm{d}}^\omega (v)$ and otherwise $j=0,m=0$. Then for this choice of $(j, m)$, $q_{j-1}\leq \tilde p_{X_{i-1}^{\ell}} \leq q_j$ and it follows that $\sigma_i^{-,\ell, j,m}\leq \sigma_i^{\ell, j,m} \leq \sigma_i^{+,\ell, j,m}$ as $R(t)$ is decreasing.}

\begin{lemma}\label{l:ESigmaLimit}
For each $(j,m)\in[J]^2$ the following limit exists:
\[
\varsigma_{j,m,J}:=\lim_{k\rightarrow\infty}\pr \left( R_1^{1,j,m}\left(\frac{(1- q_j) m}{1+(1- q_j)m}\right) > \tilde R_1^{1,j,m}\left(\frac{(1-\mathfrak r_1^{1,j,m}) \mathfrak d_{1}^{1,j,m}}{1+(1-\mathfrak r_1^{1,j,m}) \mathfrak d_{1}^{1,j,m}}\right) \right).
\]
\end{lemma}
\begin{proof}
Note that $(\mathfrak r_1^{1,j,m},\mathfrak d_{1}^{1,j,m})$ depends on $k$, it is taken from the distribution $(\PP(k),\DD(k))$.  By Lemma \ref{l:localEscapeProb}  it follows that $(\PP(k),\DD(k))$ converges to $(\PP,\DD)$ which competes the result.
\end{proof}

Define
\[
Z_i^{+,\ell, j,m}= \sigma_i^{+,\ell, j,m} \sum_{|x|>\rho} x\mathbbm{1}_{\{ w_i^{\ell,j,m}(x) =1\}}
\]
and similarly define $Z_i^{-,\ell, j,m}$ and $Z_i^{\ell, j,m}$ replacing $\sigma_i^{+,\ell, j,m}$ with $\sigma_i^{-,\ell, j,m}$ and  $\sigma_i^{\ell, j,m}$ respectively.
Recall that for each $i\in [2^k], (j ,m) \in [J]^2\cup \{(0,0)\}$ and $\ell\in[k^3]$ we defined
\[
\phi_i^{\ell, j,m}=\sum_{i'=0}^{i-1} \NN_{i'}^{\ell,j,m},
\]
and now define
\[
\hat X_i^\ell = \sum_{(j,m)\in [J]^2\cup \{(0,0)\} } \sum_{i'=1}^{\phi_i^{\ell, j,m}} Z_{i'}^{\ell, j,m}.
\]
We will show that under $\GG$ we can jointly couple the paths $X^\ell$ and $\hat X^\ell$ with high probability in the $L^q$ norm.  The coupling will be stronger at times in the main phase  than in the special phase:  in the latter phase, the paths may differ more as the walk $X^\ell$ may traverse a long edge multiple times while in $\hat X^\ell$ the corresponding jump only occurs once.
Denote the set of times in the main phase as $\II^\ell = \{1\leq i \leq 2^k:\AA_{i-1,\ell}^*=0\}$ and let $\XX^\ell=\{1\leq i \leq 2^k:\AA_{i-1,\ell} =1\}$ denote the set of times the coupling enters the special phase.  Finally define
\[
Z_{\max}^\ell = \sum_{i=1}^{2^k}\sum_{(j,m)\in [J]^2\cup \{(0,0)\} }  \sum_{x\in \mathbbm{Z}^d} |x| \mathbbm{1}_{\{ w_i^{\ell,j,m}(x) =1\}}~.
\]
which we will use as an overall bound on the total jumps.  In the following lemma we control the coupling in the main and special phases separately.  In particular, it will turn out that the coupling is $o(2^{k/(s-d)})$ in the main phase, but only $O(2^{k/(s-d)})$ in the special phase.
\begin{lemma}\label{l:couplingApproxA}
There exists $\delta_0>0$ depending on $d, s$ so that for all $\delta \in (0, \delta_0)$, there exists $\gamma, \epsilon>0$ such that
\begin{equation}\label{e:sizeII}
\pr \left( |[2^k] \setminus \II^\ell|\geq 2^{2\gamma  k}  \bigg| \GG   \right) < 2^{-\epsilon k}
\end{equation}
that,
\begin{equation}\label{e:couplingDistanceII}
\pr\left( \max_{i\in \II^\ell} 2^{-\frac{k}{s-d}} \left|\hat X_i^\ell - X_i^\ell \right| >\frac1k \bigg| \GG \right) < 2k^{-100},
\end{equation}
and
\begin{equation}\label{e:couplingDistanceNotII}
\pr\left( \max_{i\in [2^k] \setminus \II^\ell} 2^{-\frac{k}{s-d}} \left|\hat X_i^\ell - X_i^\ell \right| >\frac2k + 2^{-\frac{k}{s-d}} Z_{\max} \bigg| \GG \right) <  2k^{-100}.
\end{equation}
\end{lemma}

\begin{proof}
By our construction, on $\GG$ we have $|[2^k]-\II^\ell|\leq 2^{\gamma k+1} |\XX^\ell|$.  Now a vertex is in $\XX^\ell$ if and only if it is a new vertex where the walk encounters a long jump and hence
$|\XX^\ell|$ is stochastically dominated by the binomial random variable $B(2^k,p)$ where
\[
p=\pr\left(  \sum_{i=1}^{2^k}\sum_{(j,m)\in [J]^2\cup \{(0,0)\} } \sum_{|y|>\rho}  w_1^{\ell,j,m}(y) =1 \right) \leq \sum_{|y|>\rho} \texttt{P}(|y|)%=O(\rho^{-(s-d)})
=O(2^{-k(1+o(1))})
\]
and hence equation \eqref{e:sizeII} holds by the Azuma-Hoeffding inequality (with room to spare).

Observe that if $\hat X_i^\ell \neq \hat X_{i-1}^\ell$ then $i\in \XX$.  So suppose that $i\in \XX$ and let $v$ denote $X_{i-1}^\ell$.  On the event $\GG$ we have that $v$ is connected to a unique vertex $x$ such that $|v-x|>\rho$.  If $\tilde p_v\leq q_J$ and $\tilde{\textrm{d}}^{\omega} (v)\leq J$ then we denote $j=\min\{j:q_j>p_v\}$ and $m=\tilde{\textrm{d}}^{\omega} (v)$ and otherwise $j=0,m=0$.  Recall the notation $\iota=\sum_{i'=1}^{i-1} \NN_{i'}^{\ell,j,m}$.  Again, under $\GG$ the walk stays inside $V^*=V_v\cup V_x$ until at least time $i-1+2^{\gamma k +1}$.  At time $i-1+2^{\gamma k +1}$ it is in $V_v$ if $\sigma_\iota^{\ell, j,m}=0$ and in $V_x$ if $\sigma_\iota^{\ell, j,m}=1$.  Hence on the event $\GG$,
\begin{align*}
\left | \left(X_{i-1+2^{\gamma k+1}}^\ell - X_{i-1}^\ell \right) - \left(\hat X_{i-1+2^{\gamma k+1}}^\ell - \hat X_{i-1}^\ell \right)\right|
&=  \left | X_{i-1+2^{\gamma k+1}}^\ell - X_{i-1}^\ell - (x-v)\sigma_\iota^{\ell, j,m} \right|  \leq 2^{\delta k}
\end{align*}
by the definition of $V_v$ and $V_x$.

Further, on the event $\GG$ we have that the displacement from $X_{i-1}^\ell$ to $X_{i}^\ell$ is smaller than $\rho$ when $i\in\II^\ell$.  It follows that
\begin{align*}
\max_{i\in \II^\ell}  \left|\hat X_i^\ell - X_i^\ell \right|
& \leq \sum_{i\in \XX^\ell} \left | \left(X_{i-1+2^{\gamma k+1}}^\ell - X_{i-1}^\ell \right) - \left(\hat X_{i-1+2^{\gamma k+1}}^\ell - \hat X_{i-1}^\ell \right)\right|+ \sum_{i\in \II^\ell} \left| X_i-X_{i-1} \right|\\
&\leq  |\XX| 2^{\delta k} + \sum_{i=1}^{2^k} \left| X_i-X_{i-1} \right|\mathbbm{1}\left\{|X_i-X_{i-1}|\leq \rho \right\} ,
\end{align*}
and hence by Lemma \ref{l:sumSmallJumps} and \eqref{e:sizeII}
\[
\pr\left( \max_{i\in \II^\ell} 2^{-\frac{k}{s-d}} \left|\hat X_i^\ell - X_i^\ell \right| >\frac1k \bigg| \GG \right) < k^{-100}.
\]
Suppose $i'\in [2^k]\setminus \II^\ell$ with $i < i' \leq i+2^{\gamma k +1}$. Then given $\GG$ and using the notation above, we have that
\[
\hat X_{i'}^\ell - \hat X_{i-1}^\ell = (x-v)\sigma_\iota^{\ell, j,m}.
\]
Since $X_{i'}^\ell\in V^*$,
\[
\min\left\{ \left| X_{i'}^\ell - X_{i-1}^\ell  \right|, \left| X_{i'}^\ell - X_{i-1}^\ell -(x-v) \right| \right\} \leq   2^{\delta k}
\]
and $|x-v|\leq Z_{\max}^\ell$ so combining this with \eqref{e:couplingDistanceII}
\begin{equation*}
\pr\left( \max_{i\in [2^k] \setminus \II^\ell} 2^{-\frac{k}{s-d}} \left|\hat X_i^\ell - X_i^\ell \right| >\frac2k + 2^{-\frac{k}{s-d}} Z_{\max} \bigg| \GG \right) < 2k^{-100},
\end{equation*}
which completes the proof.
\end{proof}

We now rescale the walks to processes in $D[0,1]$: let us define, for $0\leq t \leq 1$,
\[
X^\ell(t) := 2^{-\frac{k}{s-d}} X_{\lfloor t 2^{k}\rfloor} ^\ell, \hat X^\ell(t) := 2^{-\frac{k}{s-d}} \hat X_{\lfloor t 2^{k}\rfloor} ^\ell.
\]
%and for $\gamma<\gamma'<1$ define
%\[
%X^{\ell,\gamma'}(t) := 2^{-\frac{k}{s-d}} X_{2^{k\gamma'}\lfloor t 2^{k(1-\gamma')}} ^\ell, \hat X^{\ell,\gamma'}(t) := 2^{-\frac{k}{s-d}} \hat X_{2^{k\gamma'}\lfloor t 2^{k(1-\gamma')}\rfloor} ^\ell
%\]

\begin{corollary}\label{c:couplingInequalityA}
For any $1\leq q <\infty$ we have that
\begin{equation}\label{e:LqCouplingBoundA}
\pr\left( \exists \ell\in [k^3] : \left\|\hat X^\ell(t) - X^\ell(t) \right\|_{L^q([0,1])} \geq \frac2k \right) = o\left(\frac1{k^3}\right).
\end{equation}
%We also have that for any fixed $t\in[0,1]$,
%\begin{equation}\label{e:finiteDimCouplingBoundA}
%\pr\left( \left|\hat X^\ell(t) - X^\ell(t) \right| \geq \epsilon \right) < \epsilon.
%\end{equation}
%and if $\gamma<\gamma'<1$ then
%\[
%\pr\left( \sup_{t\in[0,1]} \left|\hat X^{\ell,\gamma'}(t) - X^{\ell,\gamma'}(t) \right| \geq \epsilon \right) < \epsilon.
%\]

\end{corollary}

\begin{proof}
Since
\[
\int_0^1 | \hat X^\ell(t) - X^\ell(t)|^q dt \leq 2^{-k}|[2^k]\setminus \II^\ell| \max_{i\in [2^k] \backslash \II^{\ell}} | \hat X_i^\ell - X_i^\ell|^q + \max_{i\in \II^\ell} | \hat X_i^\ell - X_i^\ell|^q,
\]
Lemma \ref{l:couplingApproxA} implies that we have
\[
\pr\left( \int_0^1 | \hat X^\ell(t) - X^\ell(t)|^q dt >    2^{-(1-2\gamma)  k} \left(\frac1k + 2^{-\frac{k}{s-d}}  Z_{\max}\right)^q +\frac1{k^q} \bigg| \GG   \right) < 5 k^{-100}
\]
Now by Lemma \ref{l:ZmaxBound}  (to be proved next), we have that
\[
\pr\left (2^{-(1-2\gamma)  k} \left[2^{-\frac{k}{s-d}}  |Z_{\max}|\right]^q  >  \frac1k \right) = o\left(\frac1{k^6}\right)
\]
with much room to spare.  Combining the previous two equations with with Lemma \ref{l:couplingBound} and taking a union bound over $\ell\in[k^3]$ establishes equation \eqref{e:LqCouplingBoundA}.
%Now for each $t$,
%\[
%\pr\left( \lfloor t 2^{k}\rfloor \notin \II^\ell \right) \leq \sum_{i= \lfloor t 2^{k}\rfloor-2^{\gamma k+1}}^{\lfloor t 2^{k}\rfloor} \pr\left(  \sum_{i=1}^{2^k}\sum_{(j,m)\in [J]^2\cup \{(0,0)\} } \sum_{|y|>\rho}  w_i^{\ell,j,m}(y) =1 \right)  O(k^{100}2^{-(1-\gamma) k}) = o(2^{-2\epsilon' q})
%\]
%Taking a union bound and
\end{proof}

The following lemma provides a bound over  the terms of the process $w_i(x)$, and therefore of $Z_{\max}$, completing the previous corollary.
\begin{lemma}\label{l:ZmaxBound}
There exists a constant $c>0$ such that for large $n$,
\begin{equation}\label{e:processMaxBound}
\pr\left (n^{-\frac{1}{s-d}}  \sum_{i=1}^n \sum_{x\in \Z^d} |x| w_i(x)  > y\right) \leq c y^{-(s-d)}
\end{equation}
and hence the variable $Z_{\max}^\ell$ satisfies
\[
\pr\left (2^{-\frac{k}{s-d}}  |Z_{\max}^\ell|  > y \right) \leq c y^{-(s-d)}
\]
where $c=c(\alpha)$ and $c'=c'(\alpha,J)$ do not depend on $k$, $n$ or $y$.
\end{lemma}

\begin{proof}
We may take $c>1$.In this case, the inequality is trivially satisfied for $y\leq 1$ so assume $y>1$.  By the power law decay of $\texttt{P}(|x|) $ we have that for each $i$,
\[
\pr\Bigg(  \sum_{|x| > yn^{\frac{1}{s-d}}} |x|\mathbbm{1}_{\{ w_i(x) =1\}} \neq 0 \Bigg) \leq   \sum_{|x| > yn^{\frac{1}{s-d}}} \texttt{P}(|x|) =  O(y^{-(s-d)} n^{-1})
\]
and hence
\begin{equation}\label{e:ZmaxLargeJumps}
\pr\Bigg( \sum_{i=1}^{n} \sum_{|x| > yn^{\frac{1}{s-d}}} \mathbbm{1}_{\{ w_i^{\ell,j,m}(x) =1\}} \neq 0 \Bigg) \leq c_1 y^{-(s-d)}.
\end{equation}
Again using the power law decay of $\texttt{P}(|x|) $,
\begin{equation}\label{{e:ZmaxSmallJumps}}
\E\left( n^{-\frac{1}{s-d}}  \sum_{i=1}^{n}   \sum_{|x| \leq yn^{\frac{1}{s-d}}}  |x| \mathbbm{1}_{\{ w_i(x) =1\}} \right) \leq   n^{1- 1/(s-d)} \sum_{|x| \leq yn^{\frac{1}{s-d}}} |x| \texttt{P}(|x|) \leq c_2 y^{1-(s-d)}.
\end{equation}
and so by Markov's inequality
\begin{equation}\label{e:ZmaxSmallJumps}
\pr \left( n^{-\frac{1}{s-d}}   \sum_{i=1}^{n}  \sum_{|x| \leq y n^{\frac{1}{s-d}}} |x| \mathbbm{1}_{\{ w_i(x) =1\}}  \geq y \right) \leq c_2 y^{-(s-d)}.
\end{equation}
Combining equations \eqref{e:ZmaxLargeJumps} and \eqref{e:ZmaxSmallJumps} establishes equation \eqref{e:processMaxBound}.  Now we have that $Z_{\max}^\ell$ is equal in distribution to $\sum_{i=1}^{(J^2+1)2^k} \sum_{x\in \Z^d} |x| w_i(x)$ so the result follows by taking $c=c_2(J^2+1)$.
\end{proof}

\subsection{Approximating $(\hat X_i^\ell)_{i, \ell}$ by $(\mathfrak{X}_i^\ell)_{i, \ell}$}
Recall the definitions of $C^\star,\bar{C}$ and $C_{q_{j-1}, q_j, m}$ from Section \ref{S:Coupling}.
The next step is to show that $(\hat X_i^\ell)_{i, \ell}$ is well approximated by a family $(\frak{X}_i^\ell)_{i, \ell}$ where in the definition of the process, we replace the random "time change" $\phi_i^{\ell, j,m}$ by $iC_{q_{j-1},q_j,m}$.  In other words, we can assume a constant rate of new vertices of each type.
This will be a consequence of ergodic considerations, and hence this particular coupling essentially requires us to assume that $0$ is in the infinite component, or at least a very large one (which is essentially the same thing).

For any $\chi \in (0, 1)$, the event $\HH_{k,\chi}^{\ell}$ defined in Section~\ref{s:apriori} implies that
\[
\max_{1\leq t \leq 2^k} \sup_{0\leq q_1<q_2\leq 1, m\geq 1} \left| N^{\ell,[q_1,q_2],\{m\}}_t - t C_{q_1,q_2,m} \right|  \leq \chi 2^k \,
\]
and
\[
\max_{1\leq t \leq 2^k} \left| N^{\ell}_i  -\sum_{(j,m)\in [J^2]} N^{\ell,[q_{j-1},q_j],\{m\}}_i  - t \bar{C}_{q_1,q_2,m} \right|  \leq \chi 2^k \ .
\]

\begin{lemma}\label{l:newVertexRateApproximation}
For each $\epsilon>0$ there exists a $\chi>0$ such that for large enough $k$ and all $(j,m)\in[J]^2$,
\begin{multline*}
\pr\left(\#\left\{ \ell\in [k^3] : \max_{1\leq i \leq 2^k} \int_0^1 \left| \mathbbm{1}\{t \geq 2^{-k} iC_{q_{j-1},q_j,m} \}-  \mathbbm{1}\{t\geq 2^{-k} \phi_i^{\ell, j,m} \} \right| dt >\epsilon \right\} > \epsilon k^3,  \HH_{k,\chi} \right)\\ < o\left(2^{-\epsilon' k}\right),
\end{multline*}
and
\[
\pr\left(\#\left\{ \ell\in [k^3] : \left| \bar C - 2^{-k}\phi_{2^k}^{0,0,\ell}  \right|  > \epsilon \right\} > \epsilon k^3,  \HH_{k,\chi} \right) < o\left(2^{-\epsilon' k}\right).
\]
\end{lemma}

\begin{proof}
Fix $(j,m)\in[J]^2$ and recall the definitions of $N^{\ell,[q_{j-1},q_j],\{m\}}_i$ and $\phi^{\ell,j ,m}_i$.  The first is the number of vertices $v$   with escape probability satisfying $q_{j-1} \leq p_v < q_j$ and degree $m$ encountered by path $\ell$ up to time $i$.
The second is the number of new vertices with local escape probability satisfying $q_{j-1}\leq \tilde p_v < q_j$ and local degree $\tilde{\textrm{d}}^{\omega}(v)=m$ first encountered at  time $0 \leq t\leq i$ where $t\in \II^\ell$ and not encountered in a previous path.
Then for large enough $k$,
\begin{align*}
&\pr\left ( \max_{1\leq i \leq 2^k} \left|  \phi^{\ell, j,m}_i  - i C_{q_{j-1},q_j,m} \right| > 8 \chi 2^k, \HH_{k,\chi}^{\ell} \right) \\
& \qquad \leq \pr\left ( \max_{1\leq i \leq 2^k} \left| N^{\ell,[q_{j-1},q_j],\{m\}}_i -  \phi^{\ell,j ,m}_i \right| > 7 \chi 2^k, \HH_{k,\chi}^{\ell} \right)
\end{align*}
which follows from the definition of  $\HH_{k,\chi}^{\ell}$.  Further, the right hand side is bounded by
\begin{align*}
&\qquad \leq \pr \left( |[2^k] \setminus \II^\ell|\geq  2^{2 \gamma  k}, \HH_{k,\chi}^{\ell}   \right)\\
& \qquad \quad + \pr\left( \#\left\{i:\left| p_{X_{i}^\ell}-\tilde p_{X_{i}^\ell} \right| > \frac1k \right\}+\#\left\{i: \tilde{\textrm{d}}^\omega(X_{i}^\ell)\neq \textrm{d}^\omega(X_{i}^\ell) \right\}> \frac1k 2^{k}, \HH_{k,\chi}^{\ell}\right)\\
&\qquad \quad + \pr\left ( \Big|\{X_i^\ell:1\leq i \leq 2^k\} \cap \bigcup_{\ell'=1}^{\ell-1} \{X_i^{\ell'}:1\leq i \leq 2^k\} \Big| > \frac1k 2^{k}, \HH_{k,\chi}^{\ell} \right)\\
&\qquad \quad + \pr\left ( N^{\ell,[q_{j-1}-\frac1k,q_{j-1}+\frac1k],\{m\}}_i  > 2\chi 2^k, \HH_{k,\chi}^{\ell} \right)\\
&\qquad \quad + \pr\left ( N^{\ell,[q_{j}-\frac1k,q_j+\frac1k],\{m\}}_i  > 2\chi 2^k, \HH_{k,\chi}^{\ell} \right)\\
&\qquad= o(2^{-\epsilon' k}).
\end{align*}
Now the first quantity in the sum is bounded by Lemma \ref{l:couplingApproxA}, the second is bounded by Lemma \ref{l:localEscapeProb}, the third by Lemma \ref{l:feeIntersections} and the final two follow from the definition of $\HH_{k,\chi}^{\ell}$ since $\PP$ does not have an atom at $q_{j-1}$ or $q_j$.
Noting that $N^{\ell}_i  -\sum_{(j,m)\in [J^2]} N^{\ell,q_{j-1},q_j,m}_i$ and $\phi^{\ell,0,0}_i$ make up the remainder of the new vertices we similarly have
\[
\pr\left ( \max_{1\leq i \leq 2^k} \left|  \phi^{\ell, 0,0}_i  - i \bar{C} \right| > \chi 2^k, \HH_{k,\chi}^{\ell} \right)  = o(2^{-2\epsilon' k}).
\]
Taking $\chi$ small enough and taking a union bound completes the result by the definition of $\HH_{k,\chi} $.
\end{proof}

We now define an new process
\[
\mathfrak X_i^\ell := \sum_{(j,m)\in [J]^2  } \sum_{i'=1}^{i C_{q_{j-1},q_j,m}} Z_{i'}^{+,\ell, j,m}  =   \sum_{(j,m)\in [J]^2  } \sum_{i'=1}^{i C_{q_{j-1},q_j,m}} \sigma_{i'}^{+,\ell, j,m} \sum_{|x|>\rho} x\mathbbm{1}_{\{ w_{i'}^{\ell,j,m}(x) =1\}}
\]
and $\mathfrak X^\ell (t)  := 2^{-\frac{k}{s-d}}\mathfrak X_{\lfloor t 2^{k}\rfloor} ^\ell$.  Not that the $\mathfrak X_i^\ell$ are independent for different $\ell$, have independent increments and depends only on the coupling variables.  We show that, using the previous lemma, we may couple $\hat X^\ell$ and $\mathfrak X^\ell$.}

\begin{lemma}\label{l:xHatxFrakCouple}

There exists $\zeta>0$ such that for each $\epsilon>0$ there exists $\chi>0$ such that
\[
\pr\left(\#\left\{ \ell\in [k^3] : \|\hat X(t) - \mathfrak X(t) \|_{L^q[0,1]} >  \epsilon + \zeta\psi_J^{\frac1{s-d}}  \right\} > \epsilon k^3,  \HH_{k,\chi} \right) < o\left(k^{-100}\right)
\]
\end{lemma}

\begin{proof}
We have that
\begin{align*}
\|\hat X^\ell(t) - \mathfrak X^\ell(t) \|_{L^q[0, 1]} \leq \UU^\ell_1+\UU^\ell_2
\end{align*}
where
\[
\UU_1^\ell := 2^{-\frac{k}{s-d}} \bigg\|\sum_{(j,m)\in [J]^2} \sum_{i'=1}^{\phi_{\lfloor t 2^{k}\rfloor}^{\ell, j,m}} Z_{i'}^{\ell, j,m}  - \sum_{(j,m)\in [J]^2} \sum_{i'=1}^{\lfloor t 2^{k}\rfloor C_{q_{j-1},q_j,m}} Z_{i'}^{\ell, j,m} \bigg \|_{L^q[0,1]}
\]
and
\[
\UU_2^\ell := 2^{-\frac{k}{s-d}} \bigg\|\sum_{i'=1}^{\phi_{\lfloor t 2^{k}\rfloor}^{0,0,\ell}} Z_{i'}^{\ell, 0,0}  \bigg \|_{L^q[0,1]} + 2^{-\frac{k}{s-d}}\bigg\| \sum_{(j,m)\in [J]^2} \sum_{i'=1}^{\lfloor t 2^{k}\rfloor C_{q_{j-1},q_j,m}} Z_{i'}^{\ell, j,m} - Z_{i'}^{+,\ell, j,m} \bigg \|_{L^q[0,1]}.
\]
Now $\UU_1^\ell$ is bounded above by
\begin{align*}
2^{-\frac{k}{s-d}}\sum_{(j,m)\in [J]^2  } \sum_{i=1}^{2^k} \bigg\{\left[\int_0^1 \left| \mathbbm{1}\{t \geq 2^{-k} iC_{q_{j-1},q_j,m} \}-  \mathbbm{1}\{t\geq 2^{-k} \phi_i^{\ell, j,m} \} \right| dt \right]^{1/q} \sum_{x\in\Z^d} |x| \mathbbm{1}_{\{ w_{i'}^{\ell,j,m}(x) =1\}}\bigg\} \\
\leq 2^{-\frac{k}{s-d}} Z_{\max}^\ell \max_{1\leq i \leq 2^k} \left[ \int_0^1 \left| \mathbbm{1}\{t \geq 2^{-k} iC_{q_{j-1},q_j,m} \}-  \mathbbm{1}\{t\geq 2^{-k} \phi_i^{\ell, j,m} \} \right| dt\right]^{1/q}
\end{align*}
and so by Lemmas \ref{l:ZmaxBound} and \ref{l:newVertexRateApproximation} we have that
\begin{equation}\label{e:UU1Bound}
\pr\left(\#\left\{ \ell\in [k^3] :  \UU_1^\ell >\epsilon \right\} > \epsilon k^3, \HH_{k,\chi}\right) < o\left(2^{-\epsilon' k}\right).
\end{equation}
Now define
\begin{align}\label{e:defnUU3}
\UU_3^\ell &:= 2^{-\frac{k}{s-d}} \sum_{i=1}^{2\bar{C}2^{k}} \sum_{x\in\Z^d} |x| \mathbbm{1}_{\{ w_{i}^{\ell,0,0}(x) =1\}}\nonumber\\
 & \qquad + 2^{-\frac{k}{s-d}} \sum_{(j,m)\in [J]^2} \sum_{i=1}^{2^k C_{q_{j-1},q_j,m}} (\sigma_{i}^{+,\ell, j,m} - \sigma_{i}^{-,\ell, j,m}) \sum_{x\in\Z^d} |x| \mathbbm{1}_{\{ w_{i}^{\ell,j,m}(x) =1\}}\nonumber\\
 &\stackrel{d}{=} 2^{-\frac{k}{s-d}} \sum_{i=1}^{\MM_k^\ell} \sum_{x\in \Z^d} |x| w_i(x)
\end{align}
where $\stackrel{d}{=}$ denotes equality in distribution and we define $$\MM^\ell_k=2\psi_J 2^{k} + \sum_{(j,m)\in [J]^2} \sum_{i=1}^{2^k C_{q_{j-1},q_j,m}} (\sigma_{i}^{+,\ell, j,m} - \sigma_{i}^{-,\ell, j,m}),$$
where we used the fact that the $\sigma_{i}^{+,\ell, j,m}$ and $\sigma_{i}^{-,\ell, j,m}$ are independent of the $w_{i}^{\ell,j,m}$.  Provided that we have that $\phi_{2^{k}}^{\ell, 0,0} \leq 2\bar{C}2^{k}\leq 2\psi_J 2^k$ we have that $\UU_2^\ell \leq \UU_3^\ell$.  Now by the definition of $\sigma_{i}^{\pm,\ell, j,m}$,
\begin{align*}
\pr\left( \sigma_{i}^{+,\ell, j,m} \neq \sigma_{i}^{-,\ell, j,m} \right)
&\leq \pr\left( R_i^{\ell,j,m}\left(\frac{(1- q_j) m}{1+(1- q_j)m}\right) \neq  R_i^{\ell,j,m}\left(\frac{(1- q_{j-1}) m}{1+(1- q_{j-1})m}\right) \right)\\
&\leq \E R_i^{\ell,j,m}\left(\frac{(1- q_j) m}{1+(1- q_j)m}\right) -  R_i^{\ell,j,m}\left(\frac{(1- q_{j-1}) m}{1+(1- q_{j-1})m}\right)\\
&= \frac{1}{(1- q_j) m}-\frac{1}{(1- q_{j-1}) m}\\
&\leq \psi_J
\end{align*}
since the $R_i^{\ell,j,m}$ are geometric.  Hence by standard Chernoff estimates for some $c=c(J)>0$,
\begin{align*}
&\pr\left( \sum_{(j,m)\in [J]^2} \sum_{i=1}^{2^k C_{q_{j-1},q_j,m}} \sigma_{i}^{+,\ell, j,m} - \sigma_{i}^{-,\ell, j,m} \geq 2\psi_J 2^k \right) = O\left( \exp(-c2^k) \right)
\end{align*}
since $\sum_{(j,m)\in [J]^2} C_{q_{j-1},q_j,m}\leq C^\star < 1$.
It follows by Lemma \ref{l:newVertexRateApproximation} and the previous equation that for small enough $\chi>0$,
\[
\pr\left(\#\left\{ \ell\in [k^3] : \MM^\ell_k >  4\psi_J 2^k  \right\} > \frac{\epsilon k^3}{4},  \HH_{k,\chi}^{J} \right) < o\left(2^{-\epsilon' k}\right)
\]
Now by Lemma \ref{l:ZmaxBound} we have that there exists a $\zeta>0$ such that for all $\psi_J>0$ and large enough $k$,
\[
\pr\left( 2^{-\frac{k}{s-d}} \sum_{i=1}^{4\psi_J 2^k} \sum_{x\in \Z^d} |x| w_i(x) > \zeta \psi_J^{\frac1{s-d}} \right) < \epsilon/8
\]
and so again by large deviations and by equation \eqref{e:defnUU3} we have that
\begin{equation}\label{e:UU2Bound}
\pr\left(\#\left\{ \ell\in [k^3] : \UU_2^\ell >  \zeta \psi_J^{\frac1{s-d}}  \right\} > \frac{\epsilon k^3}{2}, \HH_{k,\chi} \right) < o\left(2^{-\epsilon' k}\right).
\end{equation}
Combining equations \eqref{e:UU1Bound} and \eqref{e:UU1Bound} completes the result.
\end{proof}

\subsection{Approximating $(\frak{X}_i^\ell)_{i, \ell}$ by $\alpha$-Stable Laws $\Gamma^{\ell}_{\alpha}(t)$}
{ Since $\mathfrak X^\ell_i$ is a process with i.i.d. increments, in this subsection we can  determine the scaling limit of $\mathfrak X^\ell(t)$  under the $L^q([0,1])$ topology as well as in the Skorohod topology in $D[0,1]$.  Note that for this step, we can proceed without reference to the Long Range Percolation model and as a consequence, the meaning of $\BbbP$ is completely unambiguous (i.e. independent of the measure on environments).}
Let
\[
 K_J= \sum_{(j,m)\in [J]^2  } \varsigma_{j,m,J} C_{q_{j-1},q_j,m}
\]
and $\varsigma$ is defined in Lemma \ref{l:ESigmaLimit}.
\begin{lemma}\label{l:weakFrakLimit}
For $1\leq q < \infty$, the weak limit of $\mathfrak X^\ell (t) $ in the Skorohod or $L^q$ topology is $K_J^{\frac1{\alpha}}\Gamma_{\alpha}(t)$.
\end{lemma}
\begin{proof}
By the power law decay of $\texttt{P}(|x|) $ we have that
\begin{equation}\label{e:scalingLimitXFrakSmallJumps}
\E 2^{-\frac{k}{\alpha}} \sum_{(j,m)\in [J]^2  } \sum_{i'=1}^{i C_{q_{j-1},q_j,m}} \sigma_{i'}^{+,\ell, j,m} \sum_{|x| \leq \rho} |x| \mathbbm{1}_{\{ w_{i'}^{\ell,j,m}(x) =1\}} = o(1)
\end{equation}
Now note that the random variable $\sum_{x} x \mathbbm{1}_{\{ w_i^{\ell,j,m}(x) =1\}} $ is in the domain of attraction of
$\Gamma_{\alpha}(1)$ and hence we have that  $\sigma_i^{+,\ell, j,m} \sum_{x \in \Z^d} x \mathbbm{1}_{\{ w_i^{\ell,j,m}(x) =1\}} $ is in the domain of attraction of
$(\E \sigma_i^{+,\ell, j,m} )^{\frac1{\alpha}}\Gamma_{\alpha}(1)$.  As $\varsigma_{j,m,J}= \lim_k \E \sigma_i^{+,\ell, j,m} $, this implies that
\[
2^{-\frac{k}{\alpha}}\sum_{i'=1}^{ 2^{k} t C_{q_{j-1},q_j,m}} \sigma_{i'}^{+,\ell, j,m} \sum_{x} x \mathbbm{1}_{\{ w_{i'}^{\ell,j,m}(x) =1\}}
\]
converges weakly in both the Skorohod topology on $D[0,1])$ and the $L^q([0,1])$ topology to $\varsigma_{j,m,J}^{\frac1{\alpha}}\Gamma_{\alpha}(t)$.  Since for different $(j,m)$ the sums are independent it follows that
\[
2^{-\frac{k}{\alpha}} \sum_{(j,m)\in [J]^2  } \sum_{i'=1}^{2^{-k} t C_{q_{j-1},q_j,m}} \sigma_{i'}^{+,\ell, j,m} \sum_{|x| \leq \rho} x \mathbbm{1}_{\{ w_{i'}^{\ell,j,m}(x) =1\}}
\]
converges weakly to $K_J^{\frac1{\alpha}}\Gamma_{\alpha}(t)$.  Combining this with equation \eqref{e:scalingLimitXFrakSmallJumps} completes the result.
\end{proof}

Ultimately for our coupling we need to take $J$ going to infinity and as such in the following lemma we show that $K_J$ converges as $J\to\infty$.
\begin{lemma}
The following limit exists,
\begin{equation}\label{e:defnK}
K := \lim_{J\rightarrow\infty} K_J = \lim_{J\rightarrow\infty} \sum_{(j,m)\in [J]^2  }  \varsigma_{j,m,J} C_{q_{j-1},q_j,m}
\end{equation}
with $0<K < 1$.
\end{lemma}

\begin{proof}
For $m\geq 1$ and $q\in[0,1]$ let $\eta^m(q)=C^{0,q,m}$ and
\[
\xi(q,m)=\pr\left(  R\left(\frac{(1- q) m}{1+(1- q)m}\right) > \tilde R\left(\frac{(1-\mathfrak r) \mathfrak d}{1+(1-\mathfrak r) \mathfrak d}\right)  \right)
\]
where $R$ and $\tilde R$ are independent Geometric processes defined in equation \eqref{e:geomProcess} and where $(\mathfrak r, \mathfrak d)$ is distributed according to $(\PP,\DD)$. Since since $R(t)$ is decreasing in $t$ it follows that $\xi$  is increasing in $q$.  Further define
\[
K=\sum_{m=1}^\infty \int_0^1 \xi(q,m) d\eta^m(q)
\]
which we interpret as a Riemannn-Stieltjes integral.  We remark that regardless of its size this is well defined since each summand is well defined and positive. Now if $q_{j-1}\leq q \leq q_j$ then
since $R(t)$ is a geometric process,
\begin{align*}
0 \leq \varsigma_{j,m,J} - \xi(q,m) & \leq \pr\left(R\left(\frac{(1- q_j) m}{1+(1- q_j)m}\right) \neq R\left(\frac{(1- q) m}{1+(1- q)m}\right)\right) \\
& \leq \E  R\left(\frac{(1- q_j) m}{1+(1- q_j)m}\right) - R\left(\frac{(1- q) m}{1+(1- q)m}\right)\\
&= \frac{1}{(1- q_j) m} - \frac{1}{(1- q) m} \leq \psi_J
\end{align*}
with $ \psi_J$ the error tolerance defined in \eqref{discrete}.
Hence we have that
\begin{align*}
|K_J-K| &\leq \sum_{m=1}^\infty \int_{q_J}^1 \xi(q,m) d\eta^m(q) + \sum_{m=J+1}^\infty \int_0^{q_J} \xi(q,m) d\eta^m(q)\\
&\quad +  \sum_{m=1}^J \sum_{j=1}^J  \int_{q_{j-1}}^{q_j}  \varsigma_{j,m,J} - \xi(q,m)  d\eta^m(q)\\
&\leq 2 \psi_J~,
\end{align*}
since $\sum_{m=1}^\infty \int_{q_J}^1 1 d\eta^m(q) + \sum_{m=J+1}^\infty \int_0^{q_J} 1 d\eta^m(q)=\bar C \leq \psi_J$.
The fact that $\psi_J$ converges to 0 establishes that $K_J$ converges to $K$.  Now since the walk is transient under $\BbbP_{\nu_0}$ all return probabilities are strictly less than 1.  Each walk arrives  at new vertices at a constant rate by the ergodic theorems { below} which establishes that $K>0$.  But  by definition $K \leq C^\star$, the total rate at which new vertices are encountered, which is strictly less than $1$.
\end{proof}

\subsection{{ The complete coupling} along dyadic subsequences}

{We now combine all the results of the section to prove the full coupling between the walks $X^\ell(t)$ and isotropic $\alpha$-stable L\'{e}vy motion $\Gamma^\ell(t)$.}

\begin{theorem}\label{t:totalCoupling}
For each $\epsilon>0$ and $1\leq q < \infty$ there exists $\chi>0$  such that for  $k$ sufficiently large there is a coupling of the random walks $(X^\ell (t) )_{\ell=1}^{[k^3]}$ with independent $\alpha$-stable Levy motions $(\Gamma^{\ell}(t))_{\ell=1}^{[k^3]}$ satisfying
\[
\pr\left(\#\left\{ \ell\in [k^3] : \| X^\ell(t) - K^{\frac1{s-d}}\Gamma^\ell(t) \|_{L^q([0,1])} >  \epsilon ,   \HH_{k,\chi}^{\ell}\right\} > \epsilon k^3 \right) < \frac1{k^2}
\]
\end{theorem}

\begin{proof}
If we take $\chi>0$ small enough and take $J$ large enough so that $\zeta\psi_J^{\frac1{s-d}}<\epsilon/3$  by Corollary \ref{c:couplingInequalityA} and Lemma \ref{l:xHatxFrakCouple} we have that
\begin{equation}\label{e:totalCouplingA}
\pr\left(\#\left\{ \ell\in [k^3] : \| X^\ell(t) - \frak{X}^\ell(t) \|_{L^q[0,1]} >  \frac{\epsilon}{3} ,   \HH_{k,\chi}^{\ell}\right\} > \frac{\epsilon k^3}{3} \right) < \frac1{k^3}.
\end{equation}
By  Lemma \ref{l:weakFrakLimit} we have that $\frak{X}^\ell(t)$ converges weakly to $K_J^{\frac1{s-d}}\Gamma^\ell(t)$ in $L^q [0,1])$.  Hence by Skorohod's Theorem  we can couple $\Gamma^1$ and $\mathfrak{X}^1$ so the when $k$ is sufficiently large $\pr(\|\frak{X}^1(t) - K_J^{\frac1{s-d}}\Gamma^1(t)\|_{L^q[0,1]} > \epsilon/3)<\epsilon/4$.  As the $\frak{X}^\ell(t)$ and $\Gamma^\ell(t)$ are separately independent and identically distributed, we can extend this coupling so that by Chernoff bounds,
\begin{equation}\label{e:totalCouplingB}
\pr\left(\#\left\{ \ell\in [k^3] :\|\frak{X}^{{ \ell}}(t) - K_J^{\frac1{s-d}}\Gamma^{{ \ell}}(t)\|_{L^q[0,1]} > \epsilon/3 \right\} > \frac{\epsilon k^3}{3} \right) < \frac1{k^3}.
\end{equation}
Finally since $K_J$ converges to $K$ we have that for large enough $J$, $$\pr(\|K_J^{\frac1{s-d}}\Gamma^1(t) - K^{\frac1{s-d}}\Gamma^1(t)\|_{L^q[0,1]} > \epsilon/3)<\epsilon/4$$
and hence again by Chernoff bounds
\begin{equation}\label{e:totalCouplingC}
\pr\left(\#\left\{ \ell\in [k^3] : \|K_J^{\frac1{s-d}}\Gamma^{{ \ell}}(t) - K^{\frac1{s-d}}\Gamma^{{ \ell}}(t)\|_{L^q[0,1]} > \epsilon/3 \right\} >  \frac{\epsilon k^3}{3} \right) < \frac1{k^3}.
\end{equation}
Combining equations \eqref{e:totalCouplingA}, \eqref{e:totalCouplingB} and \eqref{e:totalCouplingC} completes the proof.
\end{proof}

\section{Proof of Theorem \ref{T:Quenched}}
\label{S:Main}

We begin with a standard topological lemma of separable Banach spaces whose proof we include for completeness.

\begin{lemma}
Let $\MM$ be a complete separable metric space with distance $d(\cdot,\cdot)$ and let $\mu$ be a Borel probability measure on $\MM$ measurable with respect to the $\sigma$-algebra generated by the open subsets of $\MM$.  Then for any $\epsilon_1>0$ there exists an $\epsilon_2>0$ and  a finite collection of disjoint measurable subsets $S_1,\ldots,S_M$ such that
\begin{itemize}
\item The union contains most of the mass of the distribution $\mu\left(\cup_{i=1}^M S_i \right) > 1-\epsilon_1$,
\item The $S_i$ are not too large: $\sup_{x,y\in S_i} d(x,y) < \epsilon_1$.
\item The $S_i$ are well separated: $\inf_{i\neq j,x\in S_i,y\in S_j} d(x,y) > \epsilon_2$.
\end{itemize}
\end{lemma}

\begin{proof}
{ By separability there exists} a countable dense subset $(x_i)_{i \in \N}$. For any sequence $(r_i)_{i \in \N}, r_i \in (\epsilon_1/2, \epsilon_1)$, the family of open balls $B_{r_i}(x_i)$ covers $\MM$.
Further, it is possible to choose these $r_i$ so that $\mu(\partial B_{r_i}(x_i)) =0$ for all $i$ ($\partial B_{r_i}(x_i)$ denotes the boundary of the ball).  For such a choice, consider the usual disjoint decomposition $S'_1 = B_{r_1}(x_1)$, $S'_i = B_{r_i}(x_i) \backslash (\cup_{j=1}^{i-1} S'_{j})$.  By construction, $\mu(\text{int}\: S'_i)= \mu(S'_i)$ ({ where} $\text{int} A$ denotes the interior of $A$) and
the sequence $(\text{int} \: S'_i)_i$ gives a disjoint family of open sets whose union has full measure.
 Thus we may find $M> 0$ so that
\[
\mu(\cup_{i \leq M}\:  \text{int} \: S'_i) \geq 1- \epsilon_1/2.
\]
By continuity of $\mu$, there is $\epsilon_2> 0$ so that if
\[
S_i= \{x \in \text{int} \: S'_i: d(x, \partial S'_i) \geq \epsilon_2\}
\]
then
\[
\mu(\cup_{i \leq M} S_i) \geq 1- \epsilon_1.
\]
\end{proof}

Applying the previous lemma, for each integer $m$ choose a finite collection of disjoint measurable subsets $S_1^m,\ldots,S_{M(m)}^m$ of $L^q([0,1])$ such that
$\pr\left(K^{\frac1{s-d}}\Gamma^1(t)\in \cup_{i=1}^M S_i \right) > 1-\frac1m$ and
$\sup_{x,y\in S_i} \|x-y\|_{L^q[0,1]} < \frac1m$. { We set $v(m)=\inf_{i\neq j,x\in S_i,y\in S_j} \|x-y\|_{L^q[0,1]}$ and let} $S^{m+}_i$ denote the enlarged set
$$
S_i^{{ m}+}:= \left\{x\in L^q([0,1]):\inf_{y\in S^{m+}_i} \|x-y\|_{L^q[0,1]} \leq \min\{\frac1m, \frac13\upsilon(m)\}\right\}.
$$
By construction, these enlargements are still disjoint.  The following lemma shows that for $\BbbP_{\mu_0}$-a.e. environment, the random walk distribution places enough weight on these  enlarged sets.

\begin{lemma}\label{l:binCouple}
For each $m>0$ and $1\leq i \leq M(m)$, { and for $\mu_0$ almost every environment $\omega\in\Omega_0$,}
\[
{ \liminf_n \pr\left(n^{-\frac1{\alpha}} { X^1_{\lfloor nt\rfloor} } \in S^{m+}_i \bigg| \omega\right) - \pr\left(K^{\frac1{\alpha}}\Gamma^1(t)\in  S^{m}_i \right) \geq 0 }
\]
\end{lemma}
\begin{proof}
First observe that for any $r>0$, by the self-similarity property $\big(\frac{K}{r}\big)^{\frac1{\alpha}}\Gamma^\ell(r t)$ is equal in distribution to $K^{\frac1{\alpha}}Z^\ell(t)$ so
\begin{equation}\label{e:ssProbEquality}
\pr\left(\left(\frac{K}{r}\right)^{\frac1{\alpha}}\Gamma^\ell(r t) \in S^{m}_i\right) = \pr\left(K^{\frac1{\alpha}}\Gamma^1(t)\in  S^{m}_i \right).
\end{equation}
Now fix $\epsilon>0$ and $k$ a positive integer and set $\theta=\min\{\frac1m, \frac13 \upsilon(m)\}$.  By Lemma \ref{t:totalCoupling} we may choose $\chi>0$ small enough so that there exists a coupling of  $(X^\ell (t) )_{\ell=1}^{[k^3]}$ and $(\Gamma(t))_{\ell=1}^{[k^3]}$ satisfying
\[
\pr\left(\#\left\{ \ell\in [k^3] : \| X^\ell(t) - K^{\frac1{\alpha}}\Gamma^\ell(t) \|_{L^q([0,1])} > { 2^{\frac{1}{\alpha}+\frac1q} } \theta  ,   \HH_{k,\chi}^{\ell}\right\} > \frac{\epsilon k^3}{4}  \right) = o(\frac1{k^2}).
\]
Now for $2^{k-1}<n\leq 2^k$ we have that
\begin{align*}
\int_0^1\left| n^{-\frac{1}{\alpha}} { X^\ell_{\lfloor nt\rfloor} }  - { \Big(\frac{K 2^k}{n}\Big)^{\frac1{\alpha}} } \Gamma^\ell(\frac{n}{2^k} t) \right|^q dt
&= \left(\frac{n}{2^k}\right)^{-1}   \int_0^{n/2^k}\left| n^{-\frac{1}{\alpha}} { X^\ell_{\lfloor 2^k t\rfloor} }  - { \Big(\frac{K 2^k}{n}\Big)^{\frac1{\alpha}} }\Gamma^\ell(t) \right|^q dt\\
&\leq  \left(\frac{n}{2^k}\right)^{-1-\frac{q}{\alpha}} \int_0^{1}\left| X^\ell(t) - K^{\frac1{\alpha}}\Gamma^\ell(t) \right|^q dt.
\end{align*}
Hence { since $\frac{n}{2^k}>\frac12$,}
\[
\bigg\|2^{-\frac{{ k}}{\alpha} } { X^\ell_{\lfloor 2^kt\rfloor} }  - K^{\frac1{\alpha}}\Gamma^{\ell}(t)\bigg\|_{L^q[0,1]} <  { 2^{\frac{1}{\alpha}+\frac1q} } \theta
\]
implies that
\[
\bigg\|n^{-\frac{1}{\alpha}} { X^\ell_{\lfloor nt\rfloor} }  - { \Big(\frac{K 2^k}{n}\Big)^{\frac1{\alpha}} }
\Gamma^\ell(\frac{n}{2^k} t) \bigg\|_{L^q[0, 1]} < \theta.
\]
Thus
\[
\pr\left(\#\left\{ \ell\in [k^3] :\max_{2^{k-1}<n\leq 2^k} \bigg\|n^{-\frac{1}{\alpha}} { X^\ell_{\lfloor nt\rfloor} }  -
{ \Big(\frac{K 2^k}{n}\Big)^{\frac1{\alpha}} }
\Gamma^\ell(\frac{n}{2^k} t) \bigg\|_{L^q[0,1]} >  \theta  ,  \HH_{k,\chi}^{\ell}\right\} > \frac{\epsilon k^3}{4}  \right) = o(\frac1{k^2})
\]
and hence if $\chi<\frac{\epsilon}{12}$, then
\begin{equation}\label{e:maxCoupleBound}
\pr\left(\#\left\{ \ell\in [k^3] :\max_{2^{k-1}<n\leq 2^k} \bigg\|n^{-\frac{1}{\alpha}} { X^\ell_{\lfloor nt\rfloor} }  -
{ \Big(\frac{K 2^k}{n}\Big)^{\frac1{\alpha}} }
\Gamma^\ell(\frac{n}{2^k} t) \bigg\|_{L^q[0,1]} >  \theta  \right\} > \frac{\epsilon k^3}{3}, \HH_{k,\chi}  \right) = o(\frac1{k^2}).
\end{equation}
Now if ${ \Big(\frac{K 2^k}{n}\Big)^{\frac1{\alpha}} }\Gamma^\ell(\frac{n}{2^k} t) \in S^{m}_i$ and $\bigg\|n^{-\frac{1}{\alpha}} { X^\ell_{\lfloor nt\rfloor} }  - { \Big(\frac{K 2^k}{n}\Big)^{\frac1{\alpha}} } \Gamma^\ell(\frac{n}{2^k} t) \bigg\|_{L^q[0,1]}<\theta$ then, by definition, we have that $n^{-\frac1{\alpha}} { X^\ell_{\lfloor nt\rfloor} }  \in S^{m+}_i$.
It follows { by the triangle inequality that}
\begin{align*}
&{ \pr\left( \min_{2^{k-1}<n\leq 2^k} \pr\left( n^{-\frac1{\alpha}}{ X^1_{\lfloor nt\rfloor} } \in S^{m+}_i \bigg| \omega\right) - \pr\left(K^{\frac1{\alpha}} \Gamma^1(t)\in  S^{m}_i \right) <  - \epsilon,\HH_{k,J}\right) }\\
&\leq \pr\left(\max_{2^{k-1}<n\leq 2^k} \#\left\{\ell\in[k^3]: \bigg\|n^{-\frac{1}{\alpha}}{ X^\ell_{\lfloor nt\rfloor} }  -
{ \Big(\frac{K 2^k}{n}\Big)^{\frac1{\alpha}} }\Gamma^\ell(\frac{n}{2^k} t) \bigg\|_{L^q[0, 1]} > \theta \right\} > \frac{\epsilon k^3}{3}, \HH_{k,\chi}\right)\\
&+  \sum_{n=2^{k-1}+1}^{2^k} \pr\left(\left| { \#\left\{\ell\in[k^3]:n^{-\frac1{\alpha}} { X^\ell_{\lfloor nt\rfloor} } \in S^{m+}_i \right\} }
- \pr\left( n^{-\frac1{\alpha}}{ X^1_{\lfloor nt\rfloor} } \in S^{m+}_i \bigg| \omega\right)k^3\right| > \frac{\epsilon k^3}{3}\right)\\
&+  \sum_{n=2^{k-1}+1}^{2^k} \pr\left(\left| \#\left\{\ell\in[k^3]: { \Big(\frac{K 2^k}{n}\Big)^{\frac1{\alpha}} }
\Gamma^\ell(\frac{n}{2^k} t) \in S^{m}_i \right\} - \pr\left(K^{\frac1{\alpha}} \Gamma^1(t)\in  S^{m}_i \right)k^3\right| > \frac{\epsilon k^3}{3} \right)\\
&\leq o\left(\frac1{k^2}\right).
\end{align*}
where we bound line 2 using equation \eqref{e:maxCoupleBound} and where each term in the two sums is bounded by $e^{-ck^3}$ by Chernoff bounds for some $c>0$.  Summing over $k$ { and using the fact that $\pr(0\in\CC^\infty)>0$} we have that for all large enough $k_0$,
\[
\pr_{\mu_0}\left({ \inf_{n> 2^{k_0}} \pr\left( n^{-\frac1{\alpha}}  X^1_{\lfloor nt\rfloor}
\in S^{m+}_i \bigg| \omega\right) - \pr\left(K^{\frac1{\alpha}} \Gamma^1(t)\in  S^{m}_i \right)  < - \epsilon }
,\cap_{k>k_0}\HH_{k,\chi}\right) = o\left(\frac1{k_0}\right).
\]
By Lemma \ref{l:fullErgodic} we have that $\lim_{k\rightarrow \infty} \pr_{\mu_0}(\cap_{k'>k}\HH_{k,\chi}\big )=1$ and so
\[
{ \liminf_{n} \pr \left( n^{-\frac1{\alpha}} X^1_{\lfloor nt\rfloor} \in S^{m+}_i \bigg| \omega\right) - \pr\left(K^{\frac1{\alpha}} \Gamma^1(t)\in  S^{m}_i \right) > - \epsilon
\quad \mu_0\hbox{--a.s.} ~,}
\]
and the result follows by taking $\epsilon$ to 0.
\end{proof}

Now using the previous lemma we prove weak convergence of the measure conditioned on the environment establishing the main theorem.

{
\begin{proof}[Theorem \ref{T:Quenched}]
Fix $\epsilon>0$.  We will let $X^n$  denote $n^{-\frac1{\alpha}}\sum_{i=1}^{\lfloor nt\rfloor} X^1_i\in L^q([0,1])$ and { let and $\Gamma$ denote $K^{\frac1{\alpha}} \Gamma(t) \in L^q([0,1])$ where $K$ is defined in \eqref{e:defnK}.  To establish Theorem \ref{T:Quenched} we will show that the law of $(X^n, 0\leq t \leq 1)$ converges weakly in $L^q([0,1])$ to the law of $(\Gamma, 0\leq t \leq 1)$.}

For a bounded continuous functional on $f$ on $L^q([0,1])$ with $\|f\|_\infty\leq 1$ { we denote}
\[
f^\delta(x)= \sup_{y\in L^q:\|x-y\|\leq\delta} f(y).
\]
Since $f$ is continuous $f^\delta \rightarrow f$ point-wisely and so by the Bounded Convergence Theorem $\E f^\delta(\Gamma)$ converges to $\E f(\Gamma)$ as $\delta\rightarrow 0$.  Choose $\delta>0$ { to be small} enough so that we have that $\E f^\delta(\Gamma)- \E f(\Gamma)<\epsilon/4$ and let $m$ be large enough so that $\frac4m<\epsilon$ and $\frac3m<\delta$.  With this { it follows that,}
\[
\max_{x\in S_i^{m+},z\in S_i^m} \|x-z\|_{L^q[0,1]} \leq \frac3m <\delta
\]
and hence
\begin{equation}\label{e:localMaxBound}
\inf_{z\in S_i^m} f^\delta(z)\geq \sup_{x\in S_i^{m+}} f(x).
\end{equation}
Since $\|f\|_\infty\leq 1$ it follows that,
\begin{equation}\label{e:mainProofA}
\E f(X^n \mid \omega)
\leq  \sum_{i=1}^{M(m)} \E \left( f(X^n) \mathbbm{1}\{X^n\in S_i^{m+}\}\mid \omega \right) + \pr\left(X^n \not\in \cup_{i=1}^{M(m)} S_i^{m+}\mid\omega\right).
\end{equation}
Now by Lemma \ref{l:binCouple} we have that for $\mu_0$ almost every environment $\omega\in\Omega_0$,
\begin{align}\label{e:mainProofB}
\limsup_n  \pr\left(X^n \not\in \cup_{i=1}^{M(m)} S_i^{m+}\mid\omega\right)
&=\limsup_n  1 -\sum_{i=1}^{M(m)} \pr\left(X^n \in  S_i^{m+}\mid\omega\right)\nonumber\\
&\leq  1 - \sum_{i=1}^{M(m)}  \pr(Z\in S_i^{m}) \leq  \frac1m \quad\mu_0\hbox{-a.s.}
\end{align}
By equation \eqref{e:localMaxBound} we have that,
\begin{align}\label{e:mainProofC}
&\E \left( f(X^n) \mathbbm{1}\{X^n\in S_i^{m+}\}\mid \omega \right)\nonumber\\
&\qquad\leq  \sum_{i=1}^{M(m)} \E f^\delta(Z) \mathbbm{1}\{Z\in S_i^{m}\}
+  2 \sum_{i=1}^{M(m)} \left|\pr(X^n\in S_i^{m+}\mid \omega) - \pr(Z\in S_i^{m})\right|.
\end{align}
We bound the first term on the right hand side by
\begin{align}\label{e:mainProofD}
\sum_{i=1}^{M(m)} \E f^\delta(Z) \mathbbm{1}\{Z\in S_i^{m}\}
&\leq  \E f^{\delta}(Z) + \pr\left(Z \not\in \cup_{i=1}^{M(m)} S_i^{m}\mid\omega\right)\nonumber\\
&\leq \E f(Z) + \frac2m.
\end{align}
For the second term  by Lemma \ref{l:binCouple} we have that for $\mu_0$ almost every $\omega\in\Omega_0$,
\begin{align}\label{e:mainProofE}
&\limsup_n \sum_{i=1}^{M(m)} \left|\pr(X^n\in S_i^{m+}\mid \omega) - \pr(Z\in S_i^{m})\right| \nonumber\\
&\qquad \leq \limsup_n \sum_{i=1}^{M(m)} \pr(X^n\in S_i^{m+}\mid \omega) - \pr(Z\in S_i^{m}) \nonumber\\
&\qquad\leq 1 - \pr(Z\in \cup_{i=1}^{M(m)} S_i^m) \leq \frac1m \quad\mu_0\hbox{-a.s.}
\end{align}
Combining equations \eqref{e:mainProofA}, \eqref{e:mainProofB}, \eqref{e:mainProofC}, \eqref{e:mainProofD} and \eqref{e:mainProofE}  establishes that for $\mu_0$ almost every $\omega\in\Omega_0$
\[
\limsup_n \E f(X^n\mid \omega) \leq \E f(Z) + \frac4m\leq \E f(Z) + \epsilon \quad\mu_0\hbox{-a.s.}
\]
Similarly we have
\[
\liminf_n \E f(X^n\mid \omega) \geq \E f(Z) - \epsilon  \quad\mu_0\hbox{-a.s.} ~,
\]
and hence $ \E f(X^n\mid \omega)$ converges to $\E f(Z)$ almost surely which establishes the weak convergence in law.
\end{proof}
}

\section{Ergodic Theory}
\label{S:Ergodic}
To make sure { that the number of new vertices that each walk $X^\ell_i$ visits in the time interval $[0,2^k]$ is approximately $2^k C$ we obtain estimates using ergodic theory.  We introduce the chain defined by ``the environment seen from the particle''. This technique only applies to the walks individually.} Later (see Lemma \ref{l:feeIntersections}) we will give a quantitative estimate on the number of \textit{vertex} intersections between a pair of walks under the distribution $\BbbP_{\mu_0}$( and under $\BbbP_{\mu}$ as well.)  When combined with the ergodic theory { estimates} outlined below, we see that {with high probability} \textit{all} the walks $(X^{\ell})_{\ell \in [k^3]}$ visit the same positive density of new vertices in $[0, 2^k]$.

Let $\tau_x: \Omega \rightarrow \Omega$ denote the shift operation:  for any { edge $b \in \mathbb Z^d \times \mathbb Z^d$ we denote $\tau_x \cdot \omega(b):= \omega(b+x)$.}  By our assumption of translation invariance of the connection probabilities $p_{i, j}$, the measure $\mu$ is clearly translation invariant for all the shifts.  The Kolmogorov $0-1$ law implies that $\mu$ is ergodic with respect to the collection of shifts $\{\tau_x\}_{x \in \mathbb Z^d}$ (See below for a similar statement for $\mu_0$).

Given an initial environment $\omega\in \Omega$ and a simple random walk trajectory $X_i$ (defined relative to $\omega$), $ \tau_{X_i}: \Omega \rightarrow \Omega$ defines a (stochastic) map. Let $\omega_{i}:= \tau_{X_i}(\omega)$, with initial environment $\omega_0=\omega$. It is clear that $\omega_i$ is a Markov process with state space $\Omega$, since the underlying random walk is.  We let $Q(\omega, \textd \omega')$ denote the transition kernel for $\omega_i$ going from $\omega$ to $\omega'$.

Given an environment $\omega$, let $d(\omega)=d^\omega(0) $ denote the degree of $\omega$ at $0$.
Recall $\textd \nu(\omega) = \frac{d(\omega)}{\E_{\mu}[d(\omega)]} \textd \mu(\omega)$ and let us introduce the Hilbert space
\[
L^2(\nu)=\{ f: \Omega \rightarrow \mathbb R: \E_\nu(f^2) < \infty\},
\]
with inner product $\langle f, g\rangle :=\int \textrm{d} \nu(\omega) f(\omega) g(\omega)$.
Since $X_i$ is reversible under the weighting $\textrm{d}^{\omega}(x)$, it follows that the operator $A_i f(\omega):= f(\omega_i)$ is self adjoint with respect to $L^2(\nu)$.
We are interested in elevating the ergodic properties of $(\mu, \Omega, \FF, (\tau_x)_{x \in \Z^d})$ to ergodic properties of the chain $\omega_i$.  { We briefly indicate how this is done.}
%Through a standard construction we may embed the sequence $(a_t)_{t \in \mathbb N}$ into a sequence $(a_t)_{t \in \mathbb Z}$ so that the entire process is stationary and ergodic under $\nu$.

Let $\Omega^{2n+1}$ denote the space of (finite) two sided sequences $(\omega_{-n} \dotsc \omega_0, \dotsc, \omega_n),\: \omega_j \in \Omega$.  On $\Omega^{2n+1}$ introduce probability measure induced by $\omega_i$.  That is, for a cylinder event $A= A_{-n} \times \dotsc \times A_n$
\[
\BbbP_{2n+1}(A):= \int_{A_{-n}}\textd \nu(\omega_{-n}) \int_{A_{-n+1}} P(\omega_{-n}, \textd \omega_{-n+1}) \dotsc \int_{A_{n}}  P(\omega_{n-1}, \textd \omega_{n}).
\]

Because $\omega_i$ is stationary and reversible under $\nu$, $\{\BbbP_{2n+1}\}_{n \in \N}$ naturally identifies with a consistent family of probability measures on $\Omega^\Z$.  Equipping $\Omega^\Z$ with the Borel  $\sigma$--field $\BB$ defined by the product topology (over time and space), we then have the existence of a probability measure, which by slight abuse { of notation} we denote by  $\BbbP_\nu$, on $\Omega^\Z$ consistent with the family $\BbbP_{2n+1}$.  Moreover, if $\vT$ denotes the Bernoulli shift  (i.e. $(\mvT \uo)_i=\uo_{i+1}$ for all $\uo \in \Omega^{\Z}$), then $\BbbP_\nu$ is stationary with respect to $\vT$ and we can study its ergodic components.

Denoting $\Omega_0=\{\omega: 0 \in \CC^{\infty}(\omega)\}$, we let $\tau_v$ act on $\Omega_0$ through the "induced shift" $\sigma_v$:  For $\omega \in \Omega_0$, define $n_v(\omega):=\inf\{n: \tau_{nv} \omega \in \Omega_0\}$.  Then $\sigma_v(\omega):= \tau_{n_{v}(\omega)v}(\omega)$.  { Analogously with $(\Omega_0, \FF_0, \nu_0)$ denoting} the restriction of the probability space $(\Omega, \FF, \nu)$ to $\Omega_0$, let $\BbbP_{\nu_0}$ denote the corresponding restriction of $\BbbP_\nu$ on $\Omega_0^{\Z}$ with $\sigma$--algebra $\BB_0$.

We need the following general result, the proof of which may be found, for example, in~\cite{Berger-Biskup}.
Let $(X, \XX, \lambda, T)$ be a probability space with the invertible, measure preserving, ergodic transformation $T$.  Let $A \in \XX$ be a set of positive measure.  For $x \in X$, let $n(x)=\inf\{ k > 0: T^k(x) \in A\}$.  Then the  Poincar\'e Recurrence Theorem implies {  $\mu$ a.s. that $n(x)< \infty $.}  We define $S: X \rightarrow A$ by $S(x)= T^{n(x)}(x)$ ({ which is} well defined up to a set of measure zero).

\begin{lemma}
As a map from $A$ to $A$, $S$ is measure preserving, ergodic and invertible up to a set of $\lambda$ measure $0$.
\end{lemma}
As a consequence we have (again see \cite{Berger-Biskup})
\begin{lemma}
Fix $b \in \N$.  Let $B \in \FF_0$ such that for almost all $\omega \in B$
\[
\BbbP(\tau_{X_b}\cdot \omega \in B|\: \omega)=1.
\]
Then it follows that $B$ is a $0-1$ event under $\nu_{0}$.
\end{lemma}
Finally by a straightforward adaptation of Proposition $3.5$ from \cite{Berger-Biskup}
\begin{lemma}
\label{L:Erg}
The measure space $(\BbbP_{\nu_0}, \Omega_0^{\Z}, \BB_0)$ is ergodic with respect to the Bernoulli shift $\vT^b$ for any $b \in \N$.
\end{lemma}

Recall Birkoff's Ergodic Theorem:
\begin{theorem}[Theorem $6.2.1$ from \cite{Durrett}]
\label{BET}
Let $(X, \XX, \lambda)$ be a probability space with measure preserving transformation $T$.  Let $\II$ be the completion of the $\sigma$-field of events invariant under $T$.  Then for any $F \in L^1(X, \lambda)$,
\[
{ \frac1N} \sum_{i=1}^N F(T^i(\omega)) \rightarrow \E[F|\II] \quad \text{ $\lambda$ a.s.}
\]
\end{theorem}

{Our main application of Theorem \ref{BET} is to the number of new vertices of different ``types''.  Recall that $N_i$ is
the indicator of the event that $X_i$ is the first visit to that vertex and  $p_v$ denotes the quenched probability that a random walk started from $v$ will ever return there.
Extending the definition from earlier slightly for  subsets $\AA\subset[0,1]$ and $\MM \subset \N$ let
\[
N^{\AA,\MM}_t := \sum_{i=1}^t N_i \mathbbm{1}\left\{ p_{X_i}\in \AA, \textrm{d}^\omega(X_i)\in\MM \right\}.
\]
Let $L(\AA,\MM) \in \Omega^{\Z}$ denote the event
\[
L(\AA,\MM):=\left\{\underline \omega: \omega_0 \notin \{\omega_i: i \leq -1\}, p(\omega_0) \in\AA,\textrm{d}(\omega_0)\in\MM \right\}.
\]

\begin{lemma}\label{l:ergodicConvergence}
We have that,
\[
\sup_{\AA,\MM}\left|\frac1t N^{\AA,\MM}_t - \BbbP_{\nu_0}(L(\AA,\MM))\right| \ra 0 \quad \nu_0 \hbox{ or } \mu_0~\mathrm{a.s.}~
\]
as $t\to\infty$ where the supremum is over all Borel subsets $\AA\subset[0,1]$ and $\MM \subset \N$ and that
\[
\BbbP_{\nu_0}(L([0,1),\N))> 0.
\]
\end{lemma}

\begin{proof}
It is enough to prove the converge of $\frac1t N^{\AA,\MM}_t$ for a single pair $(\AA,\MM)$ with the extension to uniform convergence over all pairs following by discretising the space.

We cannot immediately apply Theorem~\ref{BET} as the $N_i$ are not a stationary sequence since whether a vertex is new depends on the previous $i$ steps and of course $i$ varies.  So we compare it to the number of vertices which are new in the doubly infinite walk $(X_i)_{i\in\Z}$.  Defining
\[
f(\underline \omega)=\mathbbm{1}\left\{\omega_{0} \notin \{\omega_{j}: -\infty < j <  {0} \} \right\} \mathbbm{1}\left\{ p({\omega_{0}})\in\AA,d(\omega_{0})\in\MM \right\}
\]
we have that
\[
f(\mathbf{T}^i \underline{\omega})=\mathbbm{1}\left\{\omega_{i} \notin \{\omega_{j}: -\infty <j <  {i} \}\right \} \mathbbm{1}\left\{ p({\omega_{i}})\in\AA, \textrm{d}^\omega(\omega_{i})\in \MM \right\}.
\]
Thus applying Lemma \ref{L:Erg} and Theorem \ref{BET} we have that
\begin{equation}
\frac 1t \sum_{0 < i \leq t} \mathbbm{1}\left\{\omega_{i} \notin \{\omega_{j}: j <  {i} \} \right\}\mathbbm{1}\left\{  p({\omega_{i}})\in \AA,d(\omega_{i})\in \MM \right\} \to \BbbP_{\nu_0}(L(\AA, \MM))  \: a.s.
\end{equation}

The quantity
\[
\frac1t N^{\AA,\MM}_t - \frac 1t\sum_{0 < i \leq t} \mathbbm{1}\left\{\omega_{i} \notin \{\omega_{j}: -\infty <j <  {i} \}\right \} \mathbbm{1}\left\{ p({\omega_{i}})\in\AA, \textrm{d}^\omega(\omega_{i})\in \MM \right\}
\]
is positive.  By transience and stationarity it converges to $0$ in $L^1(\BbbP_{\nu_0})$.  (Note that transience under $\BbbP_{\nu_0}$ follows immediately from Theorem \ref{T:HKLRP}. It was originally proved via electrical network methods in \cite{BLRP}.)  On the other hand if $t= b q+r$,
\[
\frac 1t N^{\AA,\MM}_t  \leq r/t+ q/t \sum_{i \leq b} \frac 1q N^{\AA,\MM}_q \circ \mathbf{T}^{qi}.
\]
For any $q$ fixed, we may apply Lemma \ref{L:Erg} and Theorem \ref{BET} (with respect to the transformation $\mathbf{T}^q$) to conclude
\[
\limsup_t \frac 1t N^{\AA,\MM}_t  \leq \frac 1q \E_{\nu_0}[  N^{\AA,\MM}_q].
\]
Hence we have
\[
\BbbP_{\nu_0}(L(\AA,\MM)) \leq \liminf_t \frac 1t N^{\AA,\MM}_t \leq \limsup_t \frac 1t N^{\AA,\MM}_t  \leq \frac 1q \E_0[  N^{\AA,\MM}_q] \: a.s.
\]
The $L^1(\BbbP_{\nu_0})$ convergence implies  $\frac 1q\E_0[  N^{\AA,\MM}_q] \ra \BbbP_{\nu_0}(L(\AA,\MM))$ as $q \ra \infty$.
Note that reversibility and transience imply
\[
\BbbP_{\nu_0}(L([0,1), \Z))>0
\]
which completes the lemma.
\end{proof}

Now recalling the definition of $\HH_{k,\chi}$ we prove Lemma \ref{l:fullErgodic}.

\begin{proof}[Proof of Lemma \ref{l:fullErgodic}]
By Lemma \ref{l:ergodicConvergence} we have that,
\[
\lim_{t\rightarrow\infty} \sup_{\AA,\MM}\left|\frac1t N^{\AA,\MM}_t - \BbbP_{\nu_0}(L(\AA,\MM))\right|  =0 \quad \mu_0 \ \hbox{a.s.}
\]
It follows that  for $\mu_0$-almost all random environments $\omega\in\Omega_0$ there exists a finite random variable $k^*(\omega)$, measurable with respect to the $\sigma$-algebra $\sigma(\omega)$, such that
\[
\pr_0 \left( \cap_{k \geq k^*(\omega)} \HH_{k,\chi}^\ell >\chi \bigg | \omega \right) >1- \chi/2.
\]
Since the walks $X_i^\ell$ are conditionally independent given $\omega$ by applying the strong law of large numbers to the random variables
\[
\mathbbm{1}\left\{ \cap_{k \geq k^*(\omega)} \HH_{k,\chi}^\ell \right\}
\]
we have that for $\mu_0$-almost all random environments $\omega\in\Omega_0$,
\[
\liminf_{L} \frac1{L^3} \#\left\{ \ell\in[L^3]: \mathbbm{1}\left\{ \cap_{k \geq k^*(\omega)} \HH_{k,\chi}^\ell \right\}=1 \right\} \geq 1-\chi/2 \quad \hbox{a.s.}
\]
The result follows by the definition of $\HH_{k,\chi}$.
\end{proof}
}

\section{{ Bounds on rare events}}
\label{S:Tech}
Let us recall the main result of \cite{CS}.  It is proved there in the continuous time case, but extends without significant change to the discrete time walk as well (and will be used here in the latter form).

\begin{theorem}[Theorem 1 of \cite{CS}]
\label{T:HKLRP}
Let us consider $s \in (d, d+2)$ for $d\geq 2$ and $s\in (1, 2)$ for $d=1$.  Assume, for simplicity, that there exists $L$ such that
\[
p_{x,y} = 1- e^{-\beta \|x-y\|_2^{-s}} \text{ for $\|x-y\|_2 \geq L$}
\]
and suppose that the $p_{x, y}$ are translation invariant and percolating. Then there exist universal constants $C_1, {\zeta}>0$ and a family of random variables $(T_x(\omega))_{x \in \Z^d}$ with the property that $T_x(\omega)< \infty$ whenever $x \in \CC^{\infty}(\omega)$, such that the following holds:
\[
P^{\omega}_{t}(x, y) \leq C_1 \textrm{deg}^{\omega}(y) t^{-d/(s-d)} \log^{\zeta} t.
\]
{for $t \geq T_x(\omega) \vee T_y(\omega)$.}
Moreover for  {$x\in\CC^\infty(\omega)$ for any $\eta> 0$, there exists $C(\eta)>0$ so that we have
\[
\mu(T_x> k|x\in\CC^\infty(\omega)) \leq  C(\eta)k^{-\eta}
\]
}
\end{theorem}

Theorem \ref{T:HKLRP} will be used to help rule out (in the probabilistic sense) various unwanted dependencies between the random walk trajectories which cause our coupling to fail.  Theorem \ref{T:HKLRP} applies only to estimates on the law of the $k^3$ walks under $\BbbP_{\nu_0}$ and $\BbbP_{\mu_0}$.  Since the coupling construction relies on $\BbbP_{\mu}$, it is convenient to extend these estimates to $\BbbP_{\mu}$.  For this purpose the following Lemma, whose proof may be found in \cite{CS} is useful:

\begin{lemma}[Lemmas 2.9 and 2.10 of \cite{CS}]
\label{L:Sizes}
Let $N> 0$ be fixed and let $n_1 > n_2 \geq \dotsc \geq n_m $ enumerate the cluster sizes inside $[-N, N]^d$ (having sampled all internal edges).  Then there exists ${ c_1, c_2, \zeta}> 0$ independent of $N$ so that
\[
\mu(n_2 >  \log^{{\zeta}} N) < { c_1e^{-c_2\log^2 N}}.
\]
Also, if $M$ denotes the largest internally connected component in $[-N,N]^d$, there is $\epsilon>0$ so that
\[
 \BbbP_{\mu} \left( 0 \leftrightarrow  B^c_{N}(0)| 0\notin M \right)
 \leq  C N^{- \epsilon}.
\]

\end{lemma}

Let us also recall a (simple) technical lemma from \cite{CS} which we will need below.  Let $D_t = \max_{0\leq u \leq t} \|X_u\|_2$ denote the Euclidean diameter of the walk at time $t$.
\begin{lemma}[Lemma $4.1$ of \cite{CS}]
\label{L:growth}
Let $1/(s-d) <p$.  Then for either the discrete or continuous time process, there exists a constant $c>0$ so that for any $x \in \Z^d$,
\[
P^{\omega}_{x}\left(D_t > c t^{p+1} \text{ infinitely often }\right) = 0
\]
$\mu \: a.s.$
Moreover, there exist constants $ c_1, c_2, c_3>0$ such that for any $T, \lambda, p,r>0$ with $p$ as above and $r< s-d$,
\[
 \BbbP(\{P^{\omega}_0(\exists t \leq T: D_t \geq c_1t^{p+1}) > c_2/T^{\lambda} \}) \leq c_3T^{\lambda+1-pr}.
\]
\end{lemma}

The implicit consequence of previous two lemmas is the following:  If we observe the process of $k^3$ walks under $\BbbP_\mu$ up to time $2^k$ and find that more than $k^{\delta_1 + \epsilon}$ vertices have been uncovered in the exploration process, then with very high probability, $0 \in \CC^{\infty}(\omega)$.

In what follows, we formulate technical lemmas controlling the behavior of the walks $(X_{i}^{\ell})_{i\in [2^k]}^{\ell \in [k^3]}$.  As we have indicated above, it is ultimately important to us that these estimates hold for the measure $\BbbP_{\mu}$.  Most often, after scaling the walk by $n^{-\alpha}$, the nontrivial statements only concern $\BbbP_{\mu_0}$.  Further, it is often convenient for us to use stationarity, only available under $\BbbP_{\nu}, \BbbP_{\nu_0}$.  For these reasons, { we give the following lemma which allows us to} transfer bounds from one of these measures to another:

\begin{lemma}
\label{l:mu}
Let $A$ be an event defined on the sample space $\Omega \times {(\Z^d)^{\N}}^{[k^3]}$
Then, there exist constants $C_1, \dotsc, C_4 <\infty $ so that we have
\begin{align*}
&\BbbP_{\mu_0}(A) \leq C_1\BbbP_{\mu}(A), \\
&\BbbP_{\nu_0}(A) \leq C_2\BbbP_{\nu}(A), \\
 &C_3^{-1}\frac{ \BbbP_{\nu_0}(A)}{-\log \BbbP_{\nu_0} (A)} \leq \BbbP_{\mu_0}(A) \leq C_3\BbbP_{\nu_0}(A), \\
&C_4^{-1}\frac{ \BbbP_{\nu}(\textrm{d}^\omega(0)\geq 1, A)}{-\log \BbbP_{\nu} (\textrm{d}^\omega(0)\geq 1, A)} \leq \BbbP_{\mu}\left(\textrm{d}^\omega(0)\geq 1, A\right) \leq C_4\BbbP_{\nu}(\textrm{d}^\omega(0)\geq 1, A).
\end{align*}
\end{lemma}
\begin{proof}
Since $s>d$ and we have assumed that $\mu$ is supercritical, the first two inequalities, as well as the upper bounds in the third and fourth, are obvious.

On the other hand, an easy calculation gives $\BbbP_{\mu}(\textrm{d}^{\omega}(0) > x) \leq B_1 \exp(-B_2x)$ for some $B_1, B_2>0$.
 Thus we have
 \begin{multline}
\BbbP_{\mu_0} (A)
\geq \BbbP_{\mu_0}(\omega \in A, \textrm{d}^{\omega}(0)<x ) \\
\geq  x^{-1} C^{-1}  \BbbP_{\nu_0}(\omega \in A, \textrm{d}^{\omega}(0)<x )
\geq x^{-1} C^{-1} [ \BbbP_{\nu_0} (A)  -  \BbbP_{\nu_0} ( \textrm{d}^{\omega}(0) \geq x)]
 \end{multline}
Using the tail bound, let us take $x= - { \frac{2}{B_2}} \log \BbbP_{\nu_0} (A)$.  We get that
 \[
\BbbP_{\mu_0} (A) \geq C' \frac{\BbbP_{\nu_0} (A)}{ -\log \BbbP_{\nu_0} (A)}
\]
The last bound follows similarly.
\end{proof}

Recall that $p(\omega)$ denotes the return probability of the origin in the environment $\omega$,  $(\PP,\DD)$ denotes joint distribution of $(p_\omega, \textrm{d}(\omega))$ under $\BbbP_\mu$ and $\rho=\rho_k=2^{\frac{k}{s-d}}/k^{200/(1-\alpha)}$.

%and let $(\PP^*,\DD^*)$ denote the joint distribution under $\nu_0???$.
%Let $(X^\ell_i)_{\ell,i\in \mathbb{N}}$ denote independent copies of the random walk started from the origin with respect to the same environment $\omega$.
{ We now recall the definitions made in Section~\ref{s:apriori} and prove a series of lemmas establishing Proposition~\ref{c:badEvents1} which provides the quantitative estimates needed for the proof of our main theorem. For $\gamma, \delta>0$, let
\begin{align*}
A(\rho)=& \{\exists v\in\mathbb{Z}^d, \omega_{0,v}=1, |v| > \rho\} \\
B(\rho, \gamma, k)=& \{ \forall v \in \Z^d \text{ such that } \omega_{0,v}=1, |v| > \rho, \text{ we have } \:v\not\in \{X_{1},\ldots,X_{2^{\gamma k+1}}\}\}\\
C( \delta, \gamma, k)=& \{\max_{0\leq t\leq 2^{\gamma k}} |X_{t}| > 2^{\delta k}\}\\
\end{align*}
Denote the event $G$ as
\[
G(\rho, \delta, \gamma, k) = A(\rho) \cap B(\rho, \gamma, k) \cap C(\delta, \gamma,k)
\]
which which we bound in the following lemma.
}
\begin{lemma}[No Quick Escapes of Local Neighborhoods Without Long Edges]
\label{N:0}
For any $\delta{ \in(0,1)}$ there exists $\gamma,\epsilon>0$ and $c$ depending on $\gamma,\epsilon$ such that
\[
\pr_{\mu}\bigg( G(\rho, \gamma, \delta, k)\bigg) < c 2^{-(1+\epsilon) k}
\]
\end{lemma}

\begin{proof}
The lemma will be proved if we can show there exists $\epsilon>0$
\be
\label{WWW1}
\BbbP_{\mu}(G(k, \rho, \gamma, \delta) |A(\rho)) < c 2^{-\epsilon k}
\ee
since it is easy to see that
\[
\BbbP_{\mu}(A(\rho)) \leq c_1(\epsilon_1) 2^{-(1-\epsilon_1)k}
\]
for any $\epsilon_1>0$.

{ Let $A^2(\rho)$ denote the event that there are two or more edges of lengths at least $\rho$ connected to $0$.  We have that,
\[
\BbbP_{\mu}\left(P^{\omega}_0\left(C( \delta, \gamma, k)) \cap B(\rho, \gamma, k)\right)\bigg|A(\rho)\right)\leq \BbbP_{\mu}(C( \delta, \gamma, k)|A(\rho)^c) + \BbbP_{\mu}(A^2(\rho)|A(\rho)).
\]
Now
\begin{equation}\label{e:A2bound}
\BbbP_{\mu}(A^2(\rho)|A(\rho))\leq 2^{-(1+o(1))k}.
\end{equation} }
But for all $k$ sufficiently large,
\[
\BbbP_{\mu}(C( \delta, \gamma, k)|A(\rho)^c)\leq 2 { \BbbP_{\mu} }(C( \delta, \gamma, k)).
\]
We apply Lemma \ref{L:growth}.  It suffices to choose the parameters $\gamma, \lambda, p, r$ so that $\epsilon'= 1/\gamma (pr-1-\lambda) >0 $, $\gamma(p+1)<\delta$.  This can be done by first taking $pr$ sufficiently large depending on $\lambda$ and then taking $\gamma$ sufficiently small depending on $\delta/(p+1)$.  Then we have
\[
\BbbP(C( \delta, \gamma, k)) \leq c_2 2^{-\gamma \lambda k}  %Do we need this term?
\]
Letting $\epsilon = \min(\gamma \lambda, \epsilon')$ proves \eqref{WWW1}.  %To finish we apply Lemma \eqref{l:mu}.
\end{proof}

For the statement of the next lemma, { recall that}
\begin{multline*}
D(\rho, k)=\\
 \{ \exists v \in \Z^d, \: \omega_{0,v}=1, |v| > \rho,  \exists J\in [2^k] \text{ such that } X_J=v \text{ and }  \:(0,v)\not\in \{(X_{i},X_{i+1})\}_{i \leq J}\}.
\end{multline*}
\begin{lemma}[Visiting Endpoints of Long Edges without Using the Edge is Unlikely]
\label{N:2}
There exists $\epsilon>0$ and a constant $c=c(\epsilon)$ such
\[
\pr_{\mu}\bigg( D(\rho, k) \bigg) < c 2^{-(1+\epsilon)k}
\]
\end{lemma}
\begin{proof}
{Consider $\BbbP_{\mu}(D(\rho, k))$ and recall the definition of $A^2(\rho)$ from in Lemma~\ref{N:0}.  Since the event $D$ requires reaching the other end of a long edge without traversing it first we have that
\begin{align*}
\BbbP_{\mu}(D(\rho, k)) &\leq \BbbP_{\mu}(A^2(\rho)) + \sum_{t=1}^{2^k} \sum_{v:|v|\geq \rho} \BbbP_{\mu}(X_t=v) \BbbP_{\mu_0}(\omega_{0,v}=1)\\
&\leq 2^{-2k(1-o(1))} + 2^k \texttt{P}(\rho) \\
&= O(2^{-(1+\epsilon)k})
\end{align*}
for sufficiently small $\epsilon$  which completes the proof.}
\end{proof}

Finally, let us consider the pair of related events
\begin{align*}
F(\rho, \gamma, k)=&\big\{ X_{i}\in\{0,v\} \text{ for some } i \text{ such that } 2^{\gamma k+1}\leq i\leq 2^k, \: \text{ and } v\in\mathbb{Z}^d \text{s.t.} \: \omega_{0,v}=1 \text{ and } \: \|v\|_2 \geq \rho \big\}\\
F^*(\rho, k)=&\big\{\exists x, v\in\mathbb{Z}^d, \: \omega_{x,v}=1, \: \|x-v\|_2 \geq \rho, \text{ and }X^1_{i}, X^2_j \in \{x,v\} \text{ for some pair } i, j \in[2^k] \big\}
\end{align*}

\begin{lemma}[Returning to Long Edges After a Transient Time is Unlikely]
\label{N:3}
For any $0<\gamma<1$, there exists $\epsilon>0$ and a constant $c=c(\epsilon)$ such that
\[
\pr_{\mu}\bigg(F(\rho, \gamma, k)\bigg) < c 2^{-(1+\epsilon) k}.
\]
\end{lemma}

\begin{proof}
{ The  lemma follows from a straightforward application of Theorem~\ref{T:HKLRP} and the fact that $\pr_{\mu}(A(\rho)) = 2^{-(1-o(1))k}$.}
\end{proof}

\begin{corollary}[Pairs of Walks Don't Intersect at Long Edges]
\label{N:4}
There exists $\epsilon>0$ and a constant $c=c(\epsilon)$ such that
\[
\pr_{\mu}\bigg(F^*(\rho, k)\bigg) < c 2^{-\epsilon k}.
\]
\end{corollary}

\begin{proof}
To prove this corollary, we work under $\BbbP_{\nu}$.  { Since the paths  $X^1, X^2$ are reversible under $\BbbP_{\nu}$  the walk}
\[
Y_t:=
\begin{cases}
X^1_{2^k-t}\text{ for }Êt \leq 2^k\\
X^2_{t-2^k} \text{ for }Êt \geq 2^k
\end{cases}
\]
is distributed as (again, under $\BbbP_{\nu}$) $(X^1_t)_{t \in [2^{k+1}]}$.  { We now} employ Lemma \ref{N:3}:
\[
\pr_{\nu}\bigg(F^*(\rho, k)\bigg) \leq c2^{-\epsilon k} + \pr_{\nu}\bigg(\exists v \in \Z^d: \|X^1_t-v\|\geq \rho,\: \omega_{X^1_t, v}=1 \text{ for some }Êt \in [2^{\gamma k}]\bigg)
\]
where we used Lemma \ref{l:mu} to translate our bounds between measures.

 The proof is finished by noting that
 \[
  \pr_{\nu}\bigg(\exists v \in \Z^d: \|X^1_t-v\|\geq \rho, \omega_{X^1_t, v}=1 \text{ for some }Êt \in [2^{\gamma k}]\bigg) \leq c 2^{\gamma k} \rho^{-(s-d) +o(1)}
 \]
and applying Lemma \ref{l:mu}.
\end{proof}

{ Recalling our definitions from Section~\ref{s:apriori} for any} any event $E \subset \Omega^{\Z}$, let $\mathbf T^{-i} \cdot E= \{\underline \omega: T^{i}\cdot \underline \omega \in E \}$. Then we denote
\begin{align*}
\scrG(\rho, \gamma, \delta, k)=& \cup _{i=0}^{2^k} \mathbf T^{-i} \cdot G(\rho \gamma, \delta, k)\\
\scrD(\rho, k)=& \cup _{i=0}^{2^k} \mathbf T^{-i}D(\rho, k)\\
\scrF(\rho, \gamma, k)=& \cup _{i=0}^{2^k} \mathbf T^{-i} \cdot F(\rho, \gamma, k).
\end{align*}
{ The following Corollary is the content of Proposition~\ref{c:badEvents1}.}
\begin{corollary}\label{c:badEvents}
Let $\delta>0$ be fixed.  Then there exist $\gamma_1, \epsilon>0$ and $c=c(\gamma_1, \epsilon)$ such that
\[
\BbbP_{\mu}\left(\scrG(\rho, \gamma, \delta, k) \cup \scrD(\rho, k)\cup \scrF(\rho, \gamma, k)\right)< c 2^{-\epsilon k}
\]
for all $0< \gamma< \gamma_1$.
\end{corollary}
\begin{proof}
This is an easy application of Lemma \ref{l:mu} and stationarity under $\BbbP_{\nu}$ along with Lemmas \ref{N:0}, \ref{N:2} and \ref{N:3}.
\end{proof}

\begin{lemma}[Not too many intersections]\label{l:feeIntersections}
There exists $\epsilon>0$ and a constant $c>0$ such that, for $\sigma \in\{ \mu_0, \mu\}$,
\[
\E_{\sigma}\left[ \left| \left\{ X_i^1:0\leq i \leq 2^k \right\} \cap \left\{ X_i^2:0\leq i \leq  \infty\right \}  \right |\right] \leq c 2^{(1-\epsilon)k}.
\]
\end{lemma}
\begin{proof}
For any $L>0$, let $B_L(0)=[-L, L]^d \cap \Z^d$. { We begin by considering the claim of the lemma under $\BbbP_{\mu_0}$.  We denote}
\be
\label{F}
F(m, n)= \E_{\mu_0}\left[ \left| \left\{ X_i^1:0\leq i \leq n \right\} \cap \left\{ X_i^2:0\leq i \leq  m\right \}  \right |\right].
\ee

{ Large deviations estimates imply that}
\[
\mu_0(\max_{x \in B_L(0)} \textrm{d}^{\omega}(x) \geq \log^{2} L) \leq C e^{- \log^{2} L},
\]
{ while} Lemma \ref{L:growth} implies
\[
 \BbbP_{\mu_0}(\{P^{\omega}_0(\exists t \leq T: D_t \geq c_1t^{p+1}) > c_2/T^{\eta} \}) \leq c_3T^{\eta+1-pr},
\]
so that for any $\epsilon_0>0$, we can find an exponent $\lambda>0$ so that
\[
\E_{\mu_0}(P^{\omega}_L(0, B_{L^\lambda}^c(0))) < L^{-(1+\epsilon_0)} .
\]
Let us apply these estimates to \eqref{F}.  We have
\[
F(m, n)\leq \E_{\mu_0}\left[ \sum_{i \leq n, j \leq m} \sum_{y \in \Z^d} P^{\omega}_{i}(0, y)P^{\omega}_{j}(0, y)\right].
\]
Reversibility of the quenched walk implies $\textrm{d}^{\omega}(0) P^{\omega}_k(0, x)= \textrm{d}^{\omega}(x) P^{\omega}_k(x, 0)$.
Thus, there exist $\eta', C_2>0$ so that
\[
F(m, n)\leq C_1 \E_{\mu_0}[  \sum_{k \leq n, l\leq m}  \log^{\eta'}(k) \log^{\eta'}(l)  P^{\omega}_{k+l}(0, 0)]+ C_2.
\]
Next we apply Theorem \ref{T:HKLRP}: there exists a random variable $T(\omega)$ with $P(T>N) \leq C(\eta''') N^{-\eta'''}$ for any $\eta'''> 0$ so that for $t \geq T(\omega)$
\[
P^{\omega}_{t}(0, 0) \leq t^{-d/(s-d)} \log^{\delta}(t).
\]
We have, for $\eta_0>0$ large enough,
\begin{multline*}
\E_{\mu_0}[  \sum_{k \leq n, l\leq m}  \log^{\eta'}(k) \log^{\eta'}(l)  P^{\omega}_{k+l}(0, 0)] \leq \\
 \E_{\mu_0}[  \sum_{T\leq k \leq n, T \leq l\leq m}  \log^{\eta_0}(k+l)  (k+l)^{-d/(s-d)}] +
\E_{\mu_0}[ T \sum_{T \leq l\leq m}  \log^{\eta_0}(T +l)  (l)^{-d/(s-d)}] \\
+\E_{\mu_0}[ T \sum_{ T \leq k \leq n}  \log^{\eta_0}(T +k)  (k)^{-d/(s-d)}]  + \E_{\mu_0}[ T ^2 \log^{\eta_0} (T)].
\end{multline*}

  By the tail bound for $T$ and since $d/(s-d)>1$, the latter three terms on the RHS are all uniformly bounded in $m,n$.

For the first term, we have
\begin{multline*}
\E_{\mu_0}[  \sum_{T\leq k \leq n, T \leq l\leq m}  \log^{\eta'}(k+l)  (k+l)^{-d/(s-d)}] \\
 \leq C(\epsilon_1) \E_{\mu_0}[  \sum_{T\leq k \leq n}k^{1-d/(s-d)+\epsilon}]  \leq C(\epsilon) n^{2-d/(s-d)+\epsilon_1}
\end{multline*}
for any $\epsilon_1>0$.  Let $m \ra \infty$. Again, since $d/(s-d)>1$ the lemma is immediate.

The statement for \[
\E_{\mu}\left[ \left| \left\{ X_i^1:0\leq i \leq 2^k \right\} \cap \left\{ X_i^2:0\leq i \leq  \infty\right \}  \right | \bigg| 0 \notin \CC^\infty(\omega) \right]
\]
follows immediately from Lemma \ref{L:Sizes}.
\end{proof}

{
\begin{lemma}[Small jumps give small contribution]
\label{l:sumSmallJumps}
For large $k\in \N$
\[
\pr_{\mu}\left( \sum_{i=1}^{2^k} \left| X_i-X_{i-1} \right|\mathbbm{1}\left\{|X_i-X_{i-1}|\leq \rho \right\}  > \frac1{2k} 2^{\frac{k}{s-d}}\right)\leq o(k^{-100})
\]
for some constant $C>0$
\end{lemma}
\begin{proof}
This is an immediate consequence Lemma \ref{l:mu} and stationarity of the environment process.

By stationarity, we have
\[
\E_{\nu}\left[\sum_{i=1}^{2^k} \left| X_i-X_{i-1} \right|\mathbbm{1}\left\{|X_i-X_{i-1}|\leq \rho \right\}\right] = 2^k \E_{\nu}\left[\left| X_1-X_{0} \right|\mathbbm{1}\left\{|X_1-X_{0}|\leq \rho \right\}\right] .
\]
Now
\[
\E_{\nu}\left[\left| X_1-X_{0} \right|\mathbbm{1}\left\{|X_1-X_{0}|\leq \rho \right\}\right] = O( \rho^{1-\alpha} ) = o( k^{-101} 2^{-k}2^{k/(s-d)})
\]
Thus Markov's inequality implies
\[
\pr_{\nu}\left( \sum_{i=1}^{2^k} \left| X_i-X_{i-1} \right|\mathbbm{1}\left\{|X_i-X_{i-1}|\leq \rho \right\}  > \frac1{2k} 2^{\frac{k}{s-d}}\right) \leq o( k^{-100}).
\]
Applying Lemma \ref{l:mu} finishes the job.
\end{proof}
}

Given $\delta, \gamma>0$, for a vertex $v$ recall that $\tilde p_v=\tilde p_v(k)$ denotes the probability that a walk started from $v$ and conditioned to stay in the set $\{u:|v-u|<2^{k\delta}\}$ returns to $v$ before time $2^{k\gamma}$ and set to 1 if $v$ has no neighbours within distance $2^{k\delta}$.  Recall that ${  \tilde{\textrm{d}}^\omega}(v):= \#\{u:\|x-u\|_2\leq \rho\}$.

\begin{lemma}\label{l:localEscapeProb}
For all $\delta \in (0, 1)$, there exists $\gamma, \epsilon>0$ so that for $1\leq i \leq 2^k$,
\[
\pr_{\mu}\left(\left| p_{X_{i}^\ell}-\tilde p_{X_{i}^\ell} \right| > \frac1k \right)\leq 2^{-2k\epsilon}
\]
and
\[
\pr_{\mu}\left({ \tilde{\textrm{d}}^\omega}(X_{i}^\ell)\neq { \textrm{d}^\omega}(X_{i}^\ell) \right)\leq 2^{-2k\epsilon}
\]
and hence
\[
\pr_{\mu}\left(\#\left\{i:\left| p_{X_{i}^\ell}-\tilde p_{X_{i}^\ell} \right| > \frac1k \right\}+\#\left\{i:{ \tilde{\textrm{d}}^\omega}(X_{i}^\ell)\neq { \textrm{d}^\omega}(X_{i}^\ell) \right\}> \frac1k 2^{k}\right)\leq 2^{-k\epsilon}.
\]
\end{lemma}
\begin{proof}
{ By Lemma \ref{l:mu} it is enough to prove the results under the measure $\BbbP_{\nu}$.}  Using { the} stationarity of $\BbbP_{\nu}$, the second bound follows by a union bound over the connection probabilities.

For the first bound, we prove the result under $\BbbP_{\nu_0}$ and then for the other measures.  Using stationarity, it is enough to consider $i=0$.
At a quenched level, given $\omega$ such that $0 \in \CC^{\infty}(\omega)$, we may couple the conditioned walk (which gives rise to $\tilde p_{0}$), denoted by $Y_t$, to an unconditioned one, denoted by $Z_t$, until the first time the unconditioned one leaves  $B_{2^{k \delta}}(0)$ { and hence we have that}
\[
| p_0(\omega)-\tilde p_0(\omega)|\leq P_0^{\omega}\left(E\right).
\]
where
\[
E= \{Z_t \text{ exits $B_{2^{k \delta}}(0)$ before time $2^{\gamma k}$ }\}
\]

{ Combining the proofs of Lemmas \ref{N:0}, \ref{N:3}, $\BbbP_{\mu_0}\left( E \right) \leq o(2^{-2\epsilon k})$ and hence
\[
\pr_{\mu_0}\left(\left| p_{X_{i}^\ell}-\tilde p_{X_{i}^\ell} \right| > \frac1k \right)\leq 2^{-2k\epsilon}
\]
}
It remains to consider the event $\{0 \notin \CC^\infty(\omega)\}$.  Let us denote by $\CC(0)$ the connected cluster of the origin.  Clearly the bound is trivial if $0$ has no nearest neighbors.  Consider
\[
 \BbbP_{\nu} \left( E, \textrm{d}^\omega(0) \geq 1 | 0 \notin \CC^{\infty}(\omega) \right).
\]
To control this quantity, we bound
\[
 \BbbP_{\nu} \left( 0 \leftrightarrow  B^c_{2^{k \delta}}(0)| 0 \notin \CC^{\infty}(\omega) \right).
\]
Recalling Lemma \ref{L:Sizes}, if $M$ denotes the largest component of $B_{2^{k \delta}}(0)$, then
\[
 \BbbP_{\nu} \left( 0\in M | 0 \notin \CC^{\infty}(\omega) \right)\leq C 2^{-\epsilon k}
\]
and
\[
 \BbbP_{\nu} \left( 0 \leftrightarrow  B^c_{2^{k \delta}}(0), 0\notin M | 0 \notin \CC^{\infty}(\omega) \right) \leq C2^{-\epsilon k}
\]
with $\epsilon= \epsilon(\delta)$.

We conclude $\BbbP_{\nu} \left( E | 0 \notin \CC^{\infty}(\omega) \right) \leq C 2^{-\epsilon k}$.
Gathering estimates together and applying Lemma \ref{l:mu} proves the result.
\end{proof}

\section{Proof of Theorem \ref{T:BM}}
\label{S:Sketch}

In this section we show that when $d=1$ and $s>0$ that the scaling limit of the walk is Brownian motion.  Recall the hypotheses of Theorem \ref{T:BM} where we assume that
\[
{\tt P}(1)=1, {\tt P}(r)=1-e^{-\beta r^{-s}} \text{ for $r \geq 2$}
\]
since we clearly need there to be an infinite component.

\subsection{Geometry of the Random Graph and Ergodic Theory}
A the notion of a cutpoint of the graph plays a key role in our analysis.
\begin{definition}
Given an environment $\omega$ and $x \in \Z$, say that $x$ is a \textit{cutpoint} for $\omega$ if $$\omega_{a, b}=0 \text{ for all $a \leq x, b \geq x$ so that $b-a \geq 2$}.$$
Let $\Xi$ denote the set of cupoints.
\end{definition}
Note that if the walk passes from the left of the cutpoint to the right it must pass through it and that it is only connected to its nearest neighbours.
By direct calculation, for $s>2, d=1$, we have that
\[
\mu(0 \in\Xi) >0.
\]
Note that this does not hold when $s\leq 2$ which results in different scaling limits.
Given an interval $[a, b]$, let $\mathfrak C_{[a, b]}$ denote the number of cutpoints $[a, b]$.  Then ergodic considerations imply
\begin{lemma}\label{l:numberCutpoints}
For any $a, b \in\R, a< b$,
\[
\frac{1}{(b-a)N} \mathfrak C_{[aN, bN]} \ra \mu(0 \in\Xi)  \qquad \mu~\hbox{a.s.}
\]
as $N\to\infty$.
\end{lemma}
Herein, it will be convenient to assume that the origin is a cutpoint.  Suppose we show that for $\mu$-almost every environment, conditioned on the origin being a cutpoint, the scaling limit of the walk started at the origin is Brownian motion.  As the distribution is invariant under shifts this implies that the walk started at any cutpoint has scaling limit Brownian motion.  Finally, since the walk reaches a cutpoint in a finite amount of time this implies a scaling limit of Brownian motion from any starting point.  This justifies conditioning the origin to be a cutpoint.  Let $\Omega_c$ denote the environments with $0\in\Xi$ and let $\mu_c$ be the induced measure on $\Omega_c$.

Let us define $c_i, i \in \Z$ as the $i$'th cutpoint from $0$ with $i$ negative to the left, $i$ positive to the right and $c_0=0$.  Note that the gaps between the cutpoints $(c_j-c_{j-1})_{j \in \Z}$ are independent and identically distributed.  This can be seen from the fact that given $c_i$ is a cutpoint, there are no edges from the left of $c_i$ to the right and the edge on the left side of $c_i$ and on the right are independent.
Similarly, the gap environments $([c_{j-1}, c_{j}])_{j \in \Z}$ also form an independent identically distributed sequence.

\begin{lemma}
The expected gap size is finite,
\[
c_j-c_{j-1} \text{ is in $L^1(\mu)$}
\]
and its mean is given by
\be\label{e:gapMean}
\E \left(c_j-c_{j-1}\right) \leq \frac{1}{\mu(0 \in\Xi)}
\ee
\end{lemma}
\begin{proof}
By the Strong Law of Large Numbers
\[
\frac{c_n}{n} = \frac1n \left[\sum_{i=1}^n c_i-c_{i-1}\right] \to \E(c_i-c_{i-1})  \qquad \mu~\hbox{a.s.}
\]
as $n\to\infty$ while by Lemma~\ref{l:numberCutpoints} we have that
\[
\frac{1}{n} \mathfrak C_{[0, N]} \ra \mu_c(0 \in\Xi)  \qquad \mu~\hbox{a.s.}
\]
as $n\to\infty$.  Combining these obsevations yields equation \eqref{e:gapMean}.
\end{proof}

Given $\omega\in\Omega_c$, let $X_{i}$ denote a simple random walk on $\omega$ started from $y\in \Z$.  The initial state is specified by $P^{\omega}_y$.
Let us define
\begin{align*}
Q^{\omega}(j, j)= &P^{\omega}_{c_j}( X_{i} \text{ returns to $c_j$ before hitting $\{c_{j-1}, c_{j+1}\}$})\\
Q^{\omega}(j, j+1)= &P^{\omega}_{c_j}( X_{j} \text{ hits $c_{j+1}$ before hitting $\{c_{j-1}, c_{j}\}$})\\
Q^{\omega}(j, j-1)= &P^{\omega}_{c_j}( X_{i} \text{ hits $c_{j-1}$ before hitting $\{c_{j}, c_{j+1}\}$})\\
Q^{\omega}(k, \ell)=&0 \text{ otherwise.}
\end{align*}
These give us the transition probabilities of the walk restricted to the times at which it is at a cutpoint.

\begin{proposition}
\label{P:Q}
We have
\[
Q^{\omega}(j, j+1)= Q^{\omega}(j+1, j)
\]
and
\[
\frac{1}{Q^{\omega}(j, j+1)} \leq 2 (c_{j+1}-c_j).
\]
Hence $\frac{1}{Q^{\omega}(j, j+1)} \in L^1(\mu)$ and
$\E_{\mu_c} \frac{1}{Q^{\omega}(j, j+1)}  \leq 2 \E_{\mu_c} c_{j+1}-c_j.$
\end{proposition}
\begin{proof}
The first statement holds by construction while the second is a simple application of the electrical network interpretation of escape probabilities on graphs \cite{DoSn}.
\end{proof}

\subsection{Quenched Functional Central Limit Theorem for the Simple Random Walk}
The next step is to define a modified walk, which we will denote by $Z_j$, which is manifestly a square integrable martingale.
To do this, we first define a sequence $(p_j)_{j \in \Z}$ by
\begin{align*}
&p_0=c_0=0\\
&p_j-p_{j-1}=\frac{1}{Q(j, j+1)} \text{ otherwise }
\end{align*}

Let $Z_i$ denote the (quenched) walk on $(p_j)_{j \in \Z}$ with transition probabilities (slight abuse of notation here)
\[
Q(p_j, p_k):=Q(j, k).
\]
By fiat, the walk $Z_i$ always starts from $p_0=c_0$.
It is then easy to check that $Z_i$ is a martingale the quadratic variation increments
\[
E^{\omega}[(Z_j-Z_{j-1})^2|\FF_{j-1}]= \frac{1}{Q(Z_{j-1}, Z_{j-1}+1)} + \frac{1}{Q(Z_{j-1}, Z_{j-1}-1)}. \]
\begin{lemma}
\label{L:bm1}
For $\mu_c$-a.e. $\omega$, the law of the process
\[
M^{\omega,n}_{t}:= \frac1{\sqrt{n}}\left(Z_{\lfloor nt\rfloor}+(tn-\lfloor tn \rfloor)(Z_{\lfloor nt\rfloor+1}-Z_{\lfloor nt\rfloor})\right)
\]
converges weakly in $C[0, 1]$ with the uniform norm to the law of
$\sqrt{K} B_t$
where $B_t$ is a standard one dimensional Brownian motion.  Here, the diffusion constant $K$ is given by
\[
K=2\E_{\mu_c}[\frac{1}{Q^{\omega}(0, 1)}].
\]
\end{lemma}
\begin{proof}
The walk is stationary with respect to the uniform distribution $(p_j)$. Hence by the Ergodic Theorem and Proposition \ref{P:Q}
\[
\lim_{n\to\infty} \frac1n \sum_{j=1}^n \E_{\mu_c}[(Z_j-Z_{j-1})^2|\FF_{j-1}] \to 2\E_{\mu_c}[\frac{1}{Q^{\omega}(0, 1)}] \qquad \mu_c \hbox{ a.s.}
\]
which shows convergence of the quadratic variation.  The result then follows by an application of martingale Central Limit Theorem (see, e.g. \cite{Durrett})
\end{proof}

Let us define $Y_i$ as the walk on $(c_j)_{j \in \Z}$ with transition probabilities given by $Q(j, k)$.  Observe that we have a natural mapping between $Y_i$ and $Z_i$.  The following is a consequence of Proposition \ref{P:Q} and the Strong Law of Large Numbers.

\begin{lemma}\label{l:pToCcomparison}
\label{L:bmet}
For $n \neq 0$ we have that
\[
|p_n|\leq 2|c_n|
\]
and
\[
\lim  \frac{1}{|j|} \left[ \frac{p_j}{\E_{\mu_c} p_1-p_0} -  \frac{c_j}{\E_{\mu_c} c_1-c_0}\right] \ra 0 \qquad \mu_c \hbox{ a.s.}
\]
as $|j|\to\infty$.
\end{lemma}

\begin{comment}
%{\magenta {\bf  I don't think we need this lemma, that the result follows from 10.6 pretty directly, would suggest removing it.}
%\begin{lemma}
%Let $\epsilon>0$ be fixed.
%We have, as $n \ra \infty$,
\[
\frac{1}{\sqrt n} \sup_{i \leq n}\left| \frac{Z_i}{\E_{\mu_c} p_1-p_0 }- \frac{Y_i}{\E_{\mu_c} c_1-c_0}\right| \mathbb{1}\left\{\sup_{i \leq n} |Y_i| \leq \frac {\sqrt n}{\epsilon} \right\} \ra 0
\]
$Q^{\omega}$ a.s.
Also
\[
Q^{\omega}\left(\sup_{i \leq n} |Y_i| \geq \frac {\sqrt n}{\epsilon}\right) = o(1)
\]
\end{lemma}
\begin{proof}
This is a simple consequence of Proposition \ref{P:Q}, Lemmas \ref{L:bm1} and  \ref{L:bmet}.
\end{proof}
}
\end{comment}

Lemma~\ref{l:pToCcomparison} establishes the scaling of the $(p_j)$ compared to the position of the cutpoints $(c_j)$.  The next corollary then follows from Lemma~\ref{L:bm1} and ~\ref{l:pToCcomparison}.
\begin{corollary}
\label{C:bm2}
For $\mu_c$ a.e. environment $\omega$, the law of the process
\[
Y^{\omega,n}_{t}:= \frac1{\sqrt{n}}\left(Y_{\lfloor nt\rfloor}+(tn-\lfloor tn \rfloor)(Y_{\lfloor nt\rfloor+1}-Y_{\lfloor nt\rfloor})\right)
\]
converges weakly (under $Q^{\omega}$) in $C[0, 1]$ with the uniform norm to the law of $\sqrt{K^*} B_t$ where
$B_t$ a standard one dimensional Brownian motion.  The constant $K^*$ is given by
\[
K^*=\frac{2\left(\E_{\mu_c}[c_1-c_0]\right)^2}{\E_{\mu_c}[p_1-p_0]}.
\]
\end{corollary}

Finally, let us return to the SRW $X_i$. Note that there is a natural coupling between $X_i$ and $Y_i$ both the origin
so that for all $i$,
\be\label{e:cutpointWalkCouple}
Y_i = X_{\tau_i}
\ee
where $\tau_i$ denotes the $i$th visit to a cutpoint by $X$.  Another application of the Ergodic Theorem implies that
\be\label{e:rateOfCutpoints}
\lim_n \frac1n \tau_n \to \mu(0\in\Xi).
\ee

\begin{comment}
{\cyan {\bf   I'm not sure that we need this Lemma.}
\begin{lemma}
\label{L:TC}
With $A=\E_{\mu_c}[\tau_1]$, for all $\epsilon> 0$,
\[
P^{\omega}( \sup_{t \in [0, 1]}| X_{\lfloor A n t \rfloor} - X_{\tau_{\lfloor n t \rfloor}}| \geq \epsilon \sqrt{n})\ra 0
\]
$\mu \: a.s.$
\end{lemma}
\begin{proof}
The claim follows from ergodic theory and a bracketing argument using  $Y_i = X_{\tau_i}$ and Corollary \ref{C:bm2} to prove discrete H\"{o}lder continuity for $X_i$.
\end{proof}
}
\end{comment}

\begin{proof}[Proof of Theorem \ref{T:BM}]
The Theorem follows directly from Lemma~\ref{l:numberCutpoints}, Corollary \ref{C:bm2} and equations~\eqref{e:cutpointWalkCouple} and~\eqref{e:rateOfCutpoints}.
\end{proof}

\bibliographystyle{plain}
\bibliography{Bib(2)}

\end{document}